\documentclass[a4paper,12pt,reqno,oneside]{amsart}

\usepackage[english]{babel} 
\usepackage[utf8]{inputenc}  
\usepackage[T1]{fontenc} 
\usepackage{amssymb, mathtools, bbm}
\usepackage{geometry}

 \usepackage[most]{tcolorbox}

\theoremstyle{plain}
\newtheorem{theorem}{Theorem}[section]
\newtheorem*{theorem*}{Theorem}
\newtheorem{lemma}[theorem]{Lemma}
\newtheorem{prop}[theorem]{Proposition}
\newtheorem{corollary}[theorem]{Corollary}

\newtheorem*{definition*}{Definition}
\newtheorem{claim}{Claim}
\newtheorem*{claim*}{Claim}
\newtheorem{fact}[theorem]{Fact}
\newtheorem{question}[theorem]{Question}
\newtheorem*{mthm}{Main Theorem}
\newtheorem{rmk}[theorem]{Remark}
\newtheorem*{subclaim*}{Subclaim}

\newcommand{\cof}{\mathrm{cf}}
\newcommand{\otp}{\mathrm{ot}}
\newcommand{\mrm}{\mathrm}
\newcommand{\mbb}{\mathbb}
\newcommand{\mcal}{\mathcal}
\newcommand{\dom}{\mathrm{dom}}
\newcommand{\ltpt}{\mathrm{Lim}}

\title[On $2$-stationarity of $\mcal{P}_{\kappa}\lambda$]{On $2$-stationarity of $\mcal{P}_{\kappa}\lambda$}
\author{Hiroshi Sakai}
\address{Graduate School of Mathematical Sciences, University of Tokyo. 3-8-1 Komaba Meguro-ku Tokyo 153-8914, Japan.}
\email{hsakai@ms.u-tokyo.ac.jp}

\author{M.~Catalina Torres}
\address{Departament de Matemàtiques i Informàtica,
Universitat de Barcelona,
Gran Via de les Corts Catalanes 585,
08007 Barcelona, Spain. / Institute of Discrete Mathematics and Geometry,
TU Wien,
Wiedner Hauptstraße 8--10,
1040 Vienna, Austria.}
\email{mactorrespa@gmail.com}

\subjclass{03E55; 03E35}
\keywords{Higher Stationary Set, Strongly Compact Cardinal}

\begin{document}

%--------------------------------------------------------------------------------------------------------------------------------------------------------------------------------------------------------------------------------------------------------

\begin{abstract}
The notion of $n$-stationary subsets of $\mcal{P}_\kappa \lambda$ for $n < \omega$ were introduced and  studied by Cody, Lambie-Hanson \& Zhang \cite{CLHZ} and Torres \cite{T}. They proved that if $\kappa$ is supercompact, then $\mcal{P}_\kappa \lambda$ is $n$-stationary in itself for all cardinals $\lambda \geq \kappa$ and all $n < \omega$. In this paper we prove that strong compactness of $\kappa$ does not imply the $2$-stationarity of $\mcal{P}_\kappa \lambda$ for cardinals $\lambda > \kappa$. We also discuss Menas' Theorem for $1$-stationary and $2$-stationary subsets of $\mcal{P}_\kappa \lambda$.
\end{abstract}

%--------------------------------------------------------------------------------------------------------------------------------------------------------------------------------------------------------------------------------------------------------

\maketitle

%--------------------------------------------------------------------------------------------------------------------------------------------------------------------------------------------------------------------------------------------------------
\section{Introduction} \label{intro}
%--------------------------------------------------------------------------------------------------------------------------------------------------------------------------------------------------------------------------------------------------------
The notion of stationary sets is one of the most fundamental notions in Set Theory and its study has a long history. The reflection of stationary sets has also been discussed in the context of large cardinals and has been studied extensively by many set theorists. For example, see Jech \cite{Je, J}.

More recently, the notion of $n$-stationary sets of ordinals, for $n < \omega$, has arisen  naturally in other areas of Mathematical Logic, such as Topological Semantics of Modal Logic and Ordinal Analysis. See Beklemishev \& Gabelaia \cite{BD} and Arai \cite{Arai}. This notion is defined iterating the notion of reflection. We recall it here.

By recursion on $n < \omega$, the notion of $n$-stationary subsets of a limit ordinal is defined as follows.
\begin{itemize}
\item A set $S \subseteq \alpha$ is $0$-\emph{stationary} in $\alpha$ if $S$ is unbounded in $\alpha$.
\item For $n > 0$, a set $S \subseteq \alpha$ is $n$-\emph{stationary} in $\alpha$ if for any $m<n$ and any $m$-stationary $T \subseteq \alpha$, there is a limit ordinal $\beta \in S$ such that $T \cap \beta$ is $m$-stationary in $\beta$.
\end{itemize}

The notion of $1$-stationary sets on ordinals, coincides with that of stationary sets. Moreover, a regular uncountable cardinal $\kappa$ is $2$-stationary in itself if and only if the stationary reflection principle holds at $\kappa$.

Bagaria \cite{B} initiated a systematic set-theoretic study of the notion of higher stationary sets. He investigated basic properties of $n$-stationary sets and generalized the notion of $n$-stationary sets to that of $\xi$-stationary sets for any infinite ordinal $\xi$. He proved that if $\kappa$ is a $\Pi^1_n$-indescribable cardinal, then $\kappa$ is $n+1$-stationary in itself. Soon after,  Bagaria, Magidor \& Sakai \cite{BMS} proved that the converse is also true in the constructible universe $L$. On the other hand, Bagaria, Magidor \& Mancilla \cite{BMM} showed that the consistency strength of the existence of an $n+1$-stationary cardinal is much weaker than that of the existence of a $\Pi^1_n$-indescribable cardinal.

Cody, Lambie-Hanson \& Zhang \cite{CLHZ} and Torres \cite{T}, generalized the notion of $n$-stationary subsets of ordinals to subsets of $\mcal{P}_\kappa A$, for $\kappa$ weakly inaccessible cardinal, $A \supseteq \kappa$, and $\mcal{P}_\kappa A = \{ x \subseteq A : |x| < \kappa \}$. The notion of $n$-stationary subsets of $\mcal{P}_\kappa A$ is then defined by recursion on $n < \omega$ as follows.
\begin{itemize}
\item A set $S\subseteq \mcal{P}_{\kappa}{A}$ is $0$-\emph{stationary} in $\mcal{P}_{\kappa}{A}$ if $S$ is unbounded in $\mcal{P}_{\kappa}{A}$, that is, $S$ is $\subseteq$-cofinal in $\mcal{P}_\kappa A$.
\item For $n > 0$, a set $S\subseteq \mcal{P}_{\kappa}{A}$ is $n$-\emph{stationary} in $\mcal{P}_{\kappa}{A}$ if for all $m<n$ and all $m$-stationary sets $T \subseteq \mcal{P}_{\kappa}{A}$, there is $x \in S$ such that $\mu := x \cap \kappa$ is a weakly inaccessible cardinal and $T\cap \mcal{P}_{\mu}{x}$ is $m$-stationary in $\mcal{P}_{\mu}{x}$.
\end{itemize}

A systematic study of this generalized notion was also started in \cite{CLHZ, T}. For example, it was proved that for any Mahlo cardinal $\kappa$ and any cardinal $\lambda \geq \kappa$, $S \subseteq \mcal{P}_\kappa \lambda$ is $1$-stationary if and only if it is strongly stationary. It was also shown that if $\kappa$ is $\lambda$-supercompact, then $\mcal{P}_\kappa \lambda$ is $n$-stationary in itself for all $n < \omega$.

The following question then naturally arises: Does the $\lambda$-strong compactness of $\kappa$ imply the $n$-stationarity of $\mcal{P}_\kappa \lambda$? In this paper, we answer this question negatively. In fact, we prove the consistency of the existence of a strongly compact cardinal $\kappa$ such that $\mcal{P}_\kappa \lambda$ is not $2$-stationary for any cardinal $\lambda > \kappa$. More precisely, we prove the following.

\begin{mthm}
Assume GCH holds, $\kappa$ is a supercompact cardinal, and there is no measurable cardinal $> \kappa$. Then, there is a forcing extension in which
\begin{itemize}
\item[(I)] $\kappa$ is strongly compact,
\item[(II)] $\mcal{P}_\kappa \lambda$ is not $2$-stationary in itself for any cardinal $\lambda > \kappa$.
\end{itemize}
\end{mthm}

We also discuss Menas' Theorem for $1$-stationary and $2$-stationary subsets of $\mcal{P}_\kappa \lambda$. Menas \cite{Me} proved that if $\kappa$ is a regular uncountable cardinal, and $\lambda$ and $\lambda '$ are cardinals with $\kappa \leq \lambda < \lambda '$, then $S \subseteq \mcal{P}_\kappa \lambda$ is stationary in $\mcal{P}_\kappa \lambda$ if and only if
\[
S \uparrow \lambda ' := \{ x' \in \mcal{P}_\kappa \lambda ' : x' \cap \lambda \in S \}
\]
is stationary in $\mcal{P}_\kappa \lambda '$. We prove that the same holds for $1$-stationary sets when $\kappa$ is a Mahlo cardinal. We also prove that, for $2$-stationary sets, the ``if'' direction of Menas' Theorem holds, whereas the ``only if'' direction consistently fails.

This paper is constructed as follows. In \S \ref{sec:preliminaries}, we present our notation and the basic facts used throughout this paper. In \S \ref{sec:higher_stat}, we review basic properties of higher stationary subsets of $\mcal{P}_\kappa \lambda$. In \S \ref{sec:Menas}, we discuss Menas' theorem for $1$-stationary and $2$-stationary sets. Finally, in \S \ref{sec:2stat_stcmpct}, we prove the main theorem.

%--------------------------------------------------------------------------------------------------------------------------------------------------------------------------------------------------------------------------------------------------------
\section{Preliminaries} \label{sec:preliminaries}
%--------------------------------------------------------------------------------------------------------------------------------------------------------------------------------------------------------------------------------------------------------

Here we present our notation and the basic facts used in this paper. For material not covered here, we refer the reader to standard textbooks in Set Theory, such as Jech \cite{Je}, Kanamori \cite{K} and Kunen \cite{Ku}. 

We begin with some miscellaneous notation in Set Theory. Suppose $x$ is a set of ordinals. Let $\otp (x)$ denote the order-type of $x$, and let $\ltpt (x)$ denote the limit points of $x$, that is, the set of all limit ordinals $\alpha \in x$ with $\sup ( x \cap \alpha ) = \alpha$. For a regular cardinal $\gamma$, we say that $x$ is $\gamma$-\emph{closed} if $\sup(y) \in x$ for every $y \subseteq x$ of order-type $\gamma$.

Let $\kappa$ be a regular uncountable cardinal. For a regular cardinal $\mu < \kappa$, let $E^\kappa_\mu$ be the set of all $\alpha < \kappa$ with $\cof ( \alpha ) = \mu$. For a set $A$, let $\mcal{P}_\kappa A$ be the set of all $x \subseteq A$ with $|x| < \kappa$. In most cases, we assume $\kappa \subseteq A$ when working with $\mcal{P}_\kappa A$.

Suppose $\kappa$ is a regular uncountable cardinal and $A \supseteq \kappa$. Let $S \subseteq \mcal{P}_\kappa A$. $S$ is said to be \emph{unbounded} in $\mcal{P}_{\kappa} A$ if for any $x \in \mcal{P}_{\kappa}{A}$ there is some $y \in S$ such that $x \subseteq y$. We say that $S$ is \emph{closed} in $\mcal{P}_{\kappa}{A}$ if $\bigcup_{\alpha < \beta} x_\alpha \in S$ for any $\beta < \kappa$ and any $\subseteq$-increasing sequence  $\langle x_\alpha  : \alpha < \beta \rangle$ in $S$. $S$ is \emph{club} in $\mcal{P}_{\kappa}{A}$ if $S$ is closed and unbounded in $\mcal{P}_{\kappa}{A}$. $S$ is said to be \emph{stationary} in $\mcal{P}_{\kappa}{A}$ if $S \cap Z \neq \emptyset$ for any club $Z \subseteq \mcal{P}_{\kappa}{A}$.

For basic facts on stationary sets, we refer the reader to Jech \cite{J}. We will often use the fact that $S$ is a stationary subset of $\mcal{P}_\kappa A$ if and only if for all $f: A^{< \omega} \rightarrow \mcal{P}_\kappa A$, there is $x \in S$ such that $f(a) \subseteq x$ for all $a \in x^{< \omega}$. We also use the fact that if $\kappa$ is weakly inaccessible, then the set $\{ x \in \mcal{P}_\kappa A : \mbox{$x \cap \kappa$ is a cardinal} \}$ is club in $\mcal{P}_\kappa A$. Note also that if $\kappa$ is inaccessible, then there are club many $x \in \mcal{P}_\kappa A$ with $x \cap \kappa$ strong limit.

In this paper, a $\kappa$-\emph{ultrafilter} over $\kappa$ means a non-principal $\kappa$-complete ultrafilter over $\kappa$. So $\kappa$ is measurable if and only if there is a $\kappa$-ultrafilter over $\kappa$. Similarly, a $\kappa$-\emph{ultrafilter} over $\mcal{P}_\kappa A$ means a non-principal fine $\kappa$-complete ultrafilter over $\mcal{P}_\kappa A$. Thus, $\kappa$ is $\lambda$-strongly compact if and only if there is a $\kappa$-ultrafilter over $\mcal{P}_\kappa \lambda$. Also, $\kappa$ is $\lambda$-supercompact if and only if there is a normal $\kappa$-ultrafilter over $\mcal{P}_\kappa \lambda$. See Jech \cite{Je} and Kanamori \cite{K} for basic facts on these notions.

By an \emph{inner model} we mean a transitive model of ZFC containing all ordinals. Let $M$ be an inner model, and suppose that in $M$, $U$ is a countably complete ultrafilter over some set. Then $\mrm{Ult} ( M , U )$ denotes the transitive collapse of the ultrapower of $M$ by $U$.

In the rest of this section, we introduce our notation and basic facts concerning forcing. We largely follow Kunen \cite{Ku} and Cummings \cite{C} for the notation of forcing.

A \emph{forcing poset} is a pre-ordered set with a largest element. Let $\mbb{P}$ be a forcing poset. The largest element of $\mbb{P}$ is denoted by $1_\mbb{P}$. For a $\mbb{P}$-name $\dot{x}$ and a $\mbb{P}$-generic filter $G$ over $V$, we let $i_G ( \dot{x} )$ denote the interpretation of $\dot{x}$ by $G$. Let $\mbb{P}'$ be another forcing poset. We say that $\mbb{P}$ and $\mbb{P}'$ are \emph{forcing equivalent} and denote as $\mbb{P} \sim \mbb{P}'$ if for some forcing poset $\mbb{Q}$, there are dense embeddings $d : \mbb{P} \to \mbb{Q}$ and $d' : \mbb{P}' \to \mbb{Q}$.

Let $\kappa$ be a cardinal. We say that $\mbb{P}$ satisfies the \emph{$\kappa$-chain condition} (or is $\kappa$-\emph{c.c.}) if every antichain in $\mbb{P}$ has size $< \kappa$. We now recall the notion of$\kappa$-strategic closure of $\mbb{P}$.

Let $\Game_\kappa ( \mbb{P} )$ be the following two-player game of length $\kappa$, played by Even and Odd .
\[
\begin{array}{c|cccccccccc}
\mbox{Even} & p_0 = 1_\mbb{P} & & p_2 & & \cdots\cdots & p_\omega & & \cdots\cdots \\
\hline
\mbox{Odd} & & p_1 & & p_3 & \cdots\cdots & & p_{\omega + 1} & \cdots\cdots
\end{array}
\]
Even and Odd in turn choose conditions $p_\alpha$ in $\mbb{P}$ to construct a descending sequence $\langle p_\alpha : \alpha < \kappa \rangle$. Even plays at even stages and Odd at odd stages. At the $0$-th stage, Even must play $p_0 = 1_\mbb{P}$. It may happen that for some limit $\alpha < \kappa$, the set $\{ p_\beta : \beta < \alpha \}$ has no lower bound in $\mbb{P}$, in which case Even cannot choose $p_\alpha$. In that case, the game is over at the stage $\alpha$, and Odd wins. Even wins if the play  continues for $\kappa$-stages, producing a descending sequence $\langle p_\alpha : \alpha < \kappa \rangle$.

$\mbb{P}$ is said to be $\kappa$-\emph{strategically closed} if Even has a winning strategy in $\Game_\kappa ( \mbb{P} )$. Note that if $\mbb{P}$ is $\kappa$-strategically closed, then a forcing with $\mbb{P}$ adds no new sequences of ordinals of length $< \kappa$.

For notation concerning forcing iterations, we follow Cummings \cite{C}. We briefly review some notation and basic facts here. See \cite{C} and Baumgartner \cite{B} for further details.

A forcing iteration of length $\gamma$ is denoted as $(\langle \mbb{P}_\alpha : \alpha \leq \gamma \rangle , \langle \dot{\mbb{Q}}_\alpha : \alpha < \gamma \rangle)$, where each $\mbb{P}_\alpha$ is a forcing poset, and each $\dot{\mbb{Q}}_\alpha$ is a $\mbb{P}_\alpha$-name for a forcing poset. As is usual, $( \langle \mbb{P}_\alpha : \alpha \leq \gamma \rangle , \langle \dot{\mbb{Q}}_\alpha : \alpha < \gamma \rangle )$ is often denoted simply by  $\mbb{P}_\gamma$. Also, the forcing relation $\Vdash_{\mbb{P}_\alpha}$ is often abbreviated as $\Vdash_\alpha$.

Suppose $( \langle \mbb{P}_\alpha : \alpha \leq \gamma \rangle , \langle \dot{\mbb{Q}}_\alpha : \alpha < \gamma \rangle )$ is a forcing iteration. For $\alpha < \gamma$, there is a $\mbb{P}_\alpha$-name $\dot{\mbb{P}}_{\alpha , \gamma}$ such that $\mbb{P}_\gamma$ and $\mbb{P}_\alpha * \dot{\mbb{P}}_{\alpha , \gamma}$ are forcing equivalent. We refer to such $\dot{\mbb{P}}_{\alpha , \gamma}$ as a \emph{tail} of $\mbb{P}_\gamma$ after $\alpha$.

In the proof of the main theorem, we will use an Easton support forcing iteration, which takes inverse limits at singular limit stages and direct limits at regular limit stages. The following is standard.

\begin{fact} \label{fact:Easton_it_basics}
Let $\gamma$ be an inaccessible cardinal and $B$ be a subset of $\gamma$ consisting of inaccessible cardinals. Suppose $( \langle \mbb{P}_\alpha : \alpha \leq \gamma \rangle , \langle \dot{\mbb{Q}}_\alpha : \alpha < \gamma \rangle )$ is an Easton support iteration with the following properties.
\begin{itemize}
\item[(i)] If $\alpha \notin B$, then $\Vdash_\alpha ``\,\mbox{$\dot{\mbb{Q}}_\alpha$ is trivial}\,"$,
\item[(ii)] If $\alpha \in B$, then $\Vdash_\alpha ``\,\mbox{$\dot{\mbb{Q}}_\alpha$ is $\alpha$-strategically closed, $\alpha^+$-c.c., and $| \dot{\mbb{Q}}_\alpha | \leq \alpha^+$}\,"$.
\end{itemize}
Then we have the following.
\begin{enumerate}
\item $\mbb{P}_\gamma$ has a dense subset of size $\leq \gamma$, and if $\gamma$ is Mahlo, then $\mbb{P}_\gamma$ is $\gamma$-c.c.
\item Suppose $\alpha < \gamma$. Let $\beta$ be $\min ( B \setminus \alpha )$ and $\dot{\mbb{P}}_{\alpha , \gamma}$ be a tail of $\mbb{P}_\gamma$ after $\alpha$. Then $\mbb{P}_\alpha$ forces $\dot{\mbb{P}}_{\alpha , \gamma}$ to be $\beta$-strategically closed.
\item If GCH holds in $V$, then $\mbb{P}_\gamma$ preserves GCH and all cofinalities.
\end{enumerate}
\end{fact}

%----------------------------------------------------------------------------------

%--------------------------------------------------------------------------------------------------------------------------------------------------------------------------------------------------------------------------------------------------------
\section{Higher stationary subsets of $\mcal{P}_\kappa A$} \label{sec:higher_stat}
%--------------------------------------------------------------------------------------------------------------------------------------------------------------------------------------------------------------------------------------------------------

In this section, we briefly review the notion of $n$-stationary subsets of $\mcal{P}_\kappa A$ and some of its consequences relevant to this paper. We begin by recalling the definition.

%----------------------------------------------------------------------------------

\begin{definition*}[Cody, Lambie-Hanson \& Zhang \cite{CLHZ}, Torres \cite{T}]
By recursion on $n < \omega$, we define the notion of $n$-stationary subsets of $\mcal{P}_\kappa A$ for a weakly inaccessible cardinal $\kappa$ and a set $A \supseteq \kappa$ as follows.
\begin{itemize}
\item $S\subseteq \mcal{P}_{\kappa}{A}$ is $0$-\emph{stationary} in $\mcal{P}_{\kappa}{A}$ if  $S$ is unbounded in $\mcal{P}_{\kappa}{A}$.
\item For $n > 0$, $S\subseteq \mcal{P}_{\kappa}{A}$ is $n$-\emph{stationary} in $\mcal{P}_{\kappa}{A}$ if for all $m<n$ and all $m$-stationary $T \subseteq \mcal{P}_{\kappa}{A}$, there is $x \in S$ such that $\mu := x \cap \kappa$ is a weakly inaccessible cardinal, and $T\cap \mcal{P}_{\mu}{x}$ is $m$-stationary in $\mcal{P}_{\mu}{x}$.
\end{itemize}
For $T\subseteq \mcal{P}_{\kappa}{A}$ and $n < \omega$, let $d_n(T)$ be the set of all $x \in \mcal{P}_{\kappa}{A}$ such that $\mu:= x \cap \kappa$ is a weakly inaccessible cardinal and $T \cap \mcal{P}_{\mu}{x}$ is $n$-stationary in $\mcal{P}_{\mu}{x}$.
\end{definition*}

%----------------------------------------------------------------------------------

Note that if $S$ is $n$-stationary in $\mcal{P}_\kappa A$, then $S$ is $m$-stationary in $\mcal{P}_\kappa A$ for all $m < n$. For basic facts about  $n$-stationary subsets of $\mcal{P}_\kappa A$, see \cite{CLHZ, T}.

In \cite{T}, it is proved that if $\kappa$ is not weakly Mahlo, then $\mcal{P}_\kappa A$ is not $1$-stationary. Hence the notion of $n$-stationary subsets of $\mcal{P}_\kappa A$ for $n \geq 1$ is vacuous unless $\kappa$ is weakly Mahlo.

For the following fact concerning $1$-stationary sets, we first recall the definition of strongly stationary sets. Let $\kappa$ be a Mahlo cardinal and let $A$ be a set including $\kappa$. $S\subseteq \mcal{P}_{\kappa}{A}$ is said to be \emph{strongly stationary} if for all $f : \mcal{P}_{\kappa}{A} \rightarrow \mcal{P}_{\kappa}{A}$, there is $x \in S$ such that $x \cap \kappa$ is inaccessible, and $\mcal{P}_{x \cap \kappa} x$ is closed under $f$, i.e.~$f[\mcal{P}_{x \cap \kappa}{x}] \subseteq \mcal{P}_{x \cap \kappa}{x}$.

\begin{fact}[\cite{CLHZ, T}] \label{fact:1-stat}
Suppose that $\kappa$ is a weakly inaccessible cardinal, $A \supseteq \kappa$ and $S \subseteq \mcal{P}_\kappa A$.
\begin{enumerate}
\item Suppose $S$ is $1$-stationary in $\mcal{P}_\kappa A$. Then $S$ is stationary in $\mcal{P}_\kappa A$, and $S \cap Z$ is $1$-stationary in $\mcal{P}_\kappa A$ for any club $Z \subseteq \mcal{P}_\kappa A$.
\item Suppose $\kappa$ is Mahlo. Then $S \subseteq \mcal{P}_{\kappa}{A}$ is $1$-stationary if and only if $S$ is strongly stationary.
\end{enumerate}
\end{fact}

%It follows from Fact \ref{fact:1-stat} (1) and the definition of $1$-stationarity that if $\kappa$ is not weakly Mahlo, then there is no $1$-stationary $S \subseteq \mcal{P}_\kappa A$. So the notion of $n$-stationary subsets of $\mcal{P}_\kappa A$ for $n \geq 1$ is vacuous unless $\kappa$ is weakly Mahlo.

Here we make a remark, which will be often used and can be easily checked.

\begin{rmk} \label{rmk:base_set}
Suppose $\kappa$ is a regular uncountable cardinal. Let $A$ and $B$ be sets with $\kappa \subseteq A , B$ and $|A| = |B|$, and let $\pi : A \to B$ be a bijection. Suppose $S \subseteq \mcal{P}_\kappa A$ and $n < \omega$. Then $S$ is stationary, strongly stationary or $n$-stationary in $\mcal{P}_\kappa A$ if and only if the set $\{ \pi [x] : x \in S \}$ is stationary, strongly stationary or $n$-stationary in $\mcal{P}_\kappa B$, respectively.
\end{rmk}

As we mentioned in the introduction, if $\kappa$ is $\lambda$-supercompact, then $\mcal{P}_\kappa \lambda$ is $n$-stationary for all $n < \omega$. In \cite{CLHZ, T}, this was shown using an additional assumption that $\lambda^{< \kappa} = \lambda$. However, this assumption is not necessary.

\begin{prop} \label{prop:spcmpct_n_stat}Suppose $\kappa$ is a $\lambda$-supercompact cardinal, where $\lambda \geq \kappa$. Let $U$ be a normal $\kappa$-ultrafilter over $\mcal{P}_\kappa \lambda$. Then, $d_n (T) \in U$ for any $n < \omega$ and any $n$-stationary $T \subseteq \mcal{P}_\kappa \lambda$. Therefore, $\mcal{P}_\kappa \lambda$ is $n$-stationary in itself for all $n < \omega$.
\end{prop}

\begin{proof}
Let $M := \mrm{Ult} (V,U)$, and let $j : V \to M$ be the ultrapower map. Note that ${}^{\lambda^{< \kappa}} M \subseteq M$ in $V$. (See Kanamori \cite{K} Proposition 22.11.) Then $j[ \lambda ] \in \mcal{P}_{j( \kappa )} j( \lambda )$ in $M$, and $j[ \lambda ] \cap j( \kappa ) = \kappa$. Moreover $\mcal{P}_\kappa j[ \lambda ]$ and its power set are absolute between $V$ and $M$. So, for any $n < \omega$ and any $S \subseteq \mcal{P}_\kappa j[ \lambda ]$, $S$ is $n$-stationary in $V$ if and only if so is in $M$.

Suppose $n < \omega$ and $T \subseteq \mcal{P}_\kappa \lambda$ is $n$-stationary. It suffices to show that $j[ \lambda ] \in j( d_n (T) )$. For this, it is enough to prove that $j(T) \cap \mcal{P}_\kappa j[ \lambda ]$ is $n$-stationary in $\mcal{P}_\kappa j[ \lambda ]$ in $M$, since $j[ \lambda ] \in \mcal{P}_{j( \kappa )} j( \lambda )$ in $M$, and $j[ \lambda ] \cap j( \kappa ) = \kappa$. Note that
\[
j(T) \cap \mcal{P}_\kappa j[ \lambda ] = \{ j(x) : x \in T \} = \{ j[x] : x \in T \} .
\]
So, by Remark \ref{rmk:base_set}, $j(T) \cap \mcal{P}_\kappa j[ \lambda ]$ is $n$-stationary in $\mcal{P}_\kappa j[ \lambda ]$ in $V$. Then, the same holds  in $M$ by the remark at the end of the previous paragraph.
\end{proof}

We conclude this section with two facts showing that $d_0(X)$ and $d_1(X)$ are $1$-stationary subsets of $\mcal{P}_\kappa A$, when $X$ is unbounded and $1$-stationary, respectively. A generalization to arbitrary $m$ can be found in \cite{CLHZ, T}; here we only use the following.

\begin{fact}[\cite{CLHZ, T}]\label{fact:1stat_easy}
Suppose $\kappa$ is a weakly inaccessible cardinal and $\kappa \subseteq A$. If $S \subseteq \mcal{P}_\kappa A$ is $1$-stationary, and $T_1 , \dots , T_m \subseteq \mcal{P}_\kappa A$ are unbounded, where $m < \omega$, then $S\cap d_0(T_1)\cap \dots \cap d_0(T_m)$ is $1$-stationary in $\mcal{P}_\kappa A$. 
\end{fact}

\begin{corollary}\label{cor:d_1stat}
Suppose $\kappa$ is a weakly inaccessible cardinal, $\kappa \subseteq A$ and $\mcal{P}_\kappa A$ is $2$-stationary. If $S \subseteq \mcal{P}_\kappa A$ is $1$-staionary, then $d_{1}(S)$ is $1$-stationary in $\mcal{P}_\kappa A$.
\end{corollary}

%-------------------------------------------------------------------------------------

\begin{proof} Let $S \subseteq \mcal{P}_\kappa A$ be $1$-stationary. To prove that $d_1 (S)$ is $1$-stationary, take an arbitrary unbounded $T \subseteq \mcal{P}_\kappa A$. We prove that $d_1(S) \cap d_0(T) \neq \emptyset$.
By Fact \ref{fact:1stat_easy}, we have that $S \cap d_0(T)$ is $1$-stationary in $\mcal{P}_\kappa A$. Then, by $2$-stationarity of $\mcal{P}_\kappa A$, we have $\emptyset \neq d_1 ( S \cap d_0 (T) ) \subseteq d_1 (S) \cap d_1( d_0 (T) ) \subseteq  d_1(S) \cap d_0(T)$.
\end{proof}

%--------------------------------------------------------------------------------------------------------------------------------------------------------------------------------------------------------------------------------------------------------
\section{Menas' Theorem for $1$-stationary and $2$-stationary sets} \label{sec:Menas}
%--------------------------------------------------------------------------------------------------------------------------------------------------------------------------------------------------------------------------------------------------------

In this section, we discuss Menas' theorem for $1$-stationary and $2$-stationary sets. We begin by recalling the original theorem and introducing the necessary notation.

Suppose $\kappa$ is a regular uncountable cardinal, $\kappa \subseteq A \subseteq B$, $S \subseteq \mcal{P}_{\kappa}A$ and $T \subseteq \mcal{P}_{\kappa}B$. The lifting $S \uparrow B$ of $S$ to $\mcal{P}_{\kappa}B$ and the projection $T \downarrow A$ of $T$ to $\mcal{P}_\kappa A$ are defined as follows.
\[
S \uparrow B := \{ y \in \mcal{P}_{\kappa}B : y \cap A \in S \}, \hspace{20pt}
T \downarrow A := \{ y \cap A : y \in T \}.
\]
 
The following theorem, due to Menas \cite{Me}, describes how stationarity behaves under projection and lifting.

\begin{theorem}[Menas' Theorem]\label{menas}
Let $\kappa$ be a regular uncountable cardinal, and let $\lambda , \lambda '$ be cardinals with $\kappa \leq \lambda < \lambda'$. Suppose $S \subseteq \mcal{P}_{\kappa}\lambda$. Then $S$ is stationary in $\mcal{P}_{\kappa}\lambda$ if and only if $S \uparrow \lambda'$ is stationary in $\mcal{P}_{\kappa}{\lambda'}$.
\end{theorem}

Let $\kappa$, $\lambda$ and $\lambda '$ be as above, and suppose that $T \subseteq \mcal{P}_\kappa \lambda '$. Then $( T \downarrow \lambda ) \uparrow \lambda ' \supseteq T$. Thus, it follows that if $T$ is stationary in $\mcal{P}_\kappa \lambda '$, then $T \downarrow \lambda$ is stationary in $\mcal{P}_\kappa \lambda$.

In Menas' theorem, we refer to the ``if'' direction and the ``only if'' direction as the downward direction and the upward direction, respectively. We study these directions for $1$-stationary and $2$-stationary sets in the following subsections. In \S \ref{sec:Menas_1stat}, we show that Menas' Theorem for $1$-stationary sets holds if $\kappa$ is Mahlo. In \S \ref{sec:Means_2stat} we  consider the case of $2$-stationary sets. We show that the downward direction holds if $\kappa$ is Mahlo, whereas the upward direction can fail even if $\kappa$ is Mahlo.

%--------------------------------------------------------------------------------------------------------------------------------------------------------------------------------------------------------------------------------------------------------
\subsection{Menas' Theorem for $1$-stationary sets} \label{sec:Menas_1stat}
%--------------------------------------------------------------------------------------------------------------------------------------------------------------------------------------------------------------------------------------------------------

In this section, we address both directions of Menas' Theorem for $1$-stationary sets. We first prove the downward direction, which holds for any $\kappa$ weakly Mahlo cardinal. We then consider the upward direction, showing that it holds when $\kappa$ is Mahlo. We do not know whether this assumption is necessary.

\begin{prop}\label{mrtl}
Let $\kappa$ be a weakly Mahlo cardinal and let $\lambda , \lambda'$ be cardinals with $\kappa \leq \lambda < \lambda '$. Suppose $S \subseteq \mcal{P}_\kappa \lambda$. If $S \uparrow \lambda '$ is $1$-stationary in $\mcal{P}_{\kappa}\lambda'$, then $S$ is $1$-stationary in $\mcal{P}_{\kappa}{\lambda}$.
\end{prop}

\begin{proof}
Suppose that $S \uparrow \lambda '$ is $1$-stationary in $\mcal{P}_{\kappa}{\lambda'}$. To see that $S$ is $1$-stationary in $\mcal{P}_\kappa \lambda$, take an arbitrary unbounded $T \subseteq \mcal{P}_\kappa \lambda$. We must find $x \in S$ such that $\mu = x \cap \kappa$ is weakly inaccessible, and $T \cap \mcal{P}_\mu x$ is unbounded.

First note that $T \uparrow \lambda '$ is unbounded in $\mcal{P}_\kappa \lambda '$. Since $S \uparrow \lambda '$ is $1$-stationary, we can take $x' \in S \uparrow \lambda '$ such that $\mu := x' \cap \kappa$ is weakly inaccessible and the set $X' := ( T \uparrow \lambda ' ) \cap \mcal{P}_\mu x'$ is unbounded. Let $x := x' \cap \lambda$. Then $x \in S$, and $x \cap \kappa = \mu$ is weakly inaccessible. Let $X := T \cap \mcal{P}_\mu x$. Note that $X' = X \uparrow x'$. Then it follows from the unboundedness of $X'$ that $X$ is unbounded in $\mcal{P}_\mu x$. So $x$ is as desired.
\end{proof}  %---------------------------------------------------------------

We now turn to the upward direction. We first consider the case $\lambda^{<\kappa} = \lambda$.

\begin{lemma}\label{mltrr}
Let $\kappa$ be a weakly Mahlo cardinal and let $\lambda , \lambda '$ be cardinals with $\kappa \leq \lambda < \lambda'$ and $\lambda^{<\kappa}=\lambda$. Suppose $S \subseteq \mcal{P}_{\kappa}\lambda$. If $S$ is $1$-stationary in $\mcal{P}_{\kappa}\lambda$, then $S \uparrow \lambda'$ is $1$-stationary in $\mcal{P}_{\kappa}{\lambda'}$.
\end{lemma}

%----------------------------------------------------------------------------------

\begin{proof}
Suppose $S$ is $1$-stationary in $\mcal{P}_\kappa \lambda$, and $T \subseteq \mcal{P}_\kappa \lambda '$ is unbounded. We find $x' \in S \uparrow \lambda '$ such that $x' \cap \kappa$ is weakly inaccessible, and $T \cap \mcal{P}_{x' \cap \kappa} x'$ is unbounded.

Since $\lambda^{< \kappa}= \lambda$, we can easily take $A \subseteq \lambda '$ such that $|A| = \lambda \subseteq A$ and $T \cap \mcal{P}_\kappa A$ is unbounded. Take a bijection $\pi: \lambda \rightarrow A$. Then $S^* := \{ \pi [x] : x \in S \}$ is $1$-stationary in $\mcal{P}_{\kappa}{A}$ by Remark \ref{rmk:base_set}. Here note that $Z := \{ z \in \mcal{P}_\kappa A :\pi^{-1} [z] = z \cap \lambda \}$ is club in $\mcal{P}_\kappa A$. Then $S^{**} := S^* \cap Z$ is $1$-stationary in $\mcal{P}_\kappa A$ by Fact \ref{fact:1-stat}.

Since $T \cap \mcal{P}_\kappa A$ is unbounded, and $S^{**}$ is $1$-stationary in $\mcal{P}_\kappa A$, we can take $x' \in S^{**}$ such that $x' \cap \kappa$ is weakly inaccessible, and $T \cap \mcal{P}_{x' \cap \kappa} x'$ is unbounded. Moreover, since $x' \in S^{**} = S^* \cap Z$, we have $x' \cap \lambda = \pi^{-1} [x'] \in S$. So $x' \in S \uparrow A \subseteq S \uparrow \lambda '$. Thus $x'$ is as desired.
\end{proof}

Next, we consider the lifting of $1$-stationary subsets of $\mcal{P}_\kappa \lambda$ to $\mcal{P}_\kappa ( \lambda^{< \kappa} )$. For this we assume that $\kappa$ is Mahlo.

\begin{lemma} \label{lift_lambda_lambda<kappa}
Let $\kappa$ be a Mahlo cardinal and let $\lambda$ be a cardinal $\geq \kappa$. Suppose $S \subseteq \mcal{P}_\kappa \lambda$. If $S$ is $1$-stationary in $\mcal{P}_\kappa \lambda$, then $S \uparrow \lambda^{< \kappa}$ is $1$-stationary in $\mcal{P}_\kappa ( \lambda^{< \kappa} )$.
\end{lemma}

\begin{proof}
Let $\nu := \lambda^{< \kappa}$. Suppose $S$ is strongly stationary in $\mcal{P}_\kappa \lambda$. We prove that $S \uparrow \nu$ is strongly stationary in $\mcal{P}_\kappa \nu$. Take an arbitrary $h : \mcal{P}_\kappa \nu \to \mcal{P}_\kappa \nu$. We find $y \in S \uparrow \nu$ such that $y \cap \kappa$ is inaccessible, and $\mcal{P}_{y \cap \kappa} y$ is closed under $h$.

By shrinking $S$ if necessary, we may assume that $x \cap \kappa$ is inaccessible for all $x \in S$. Since $\kappa$ is inaccessible, by replacing $h(w)$ with $\bigcup_{w' \subseteq w} h(w')$ for each $w \in \mcal{P}_\kappa \nu$, we may also assume that $h$ is $\subseteq$-increasing, i.e.~if $w' \subseteq w$, then $h(w') \subseteq h(w)$.

Take a bijection $\pi : \mcal{P}_\kappa \lambda \to \nu$. For each $x \in S$ let $\pi^* (x) := \pi [ \mcal{P}_{x \cap \kappa} x ]$. Note that $| \pi^* (x) | < \kappa$ since $\kappa$ is inaccessible. So $\pi^* (x) \in \mcal{P}_\kappa \nu$. We will find $x \in S$ such that $\pi^* (x) \cap \lambda = x$, and $\mcal{P}_{x \cap \kappa} \pi^* (x)$ is closed under $h$. Then $y := \pi^* (x)$ will be as desired.

We define functions $f_0 , f_1 , f_2 : \mcal{P}_\kappa \lambda \to \mcal{P}_\kappa \lambda$ so that if $x \in S$, and $\mcal{P}_{x \cap \kappa} x$ is closed under all $f_0 , f_1 , f_2$, then $x$ is as desired.

First define $f_0 : \mcal{P}_\kappa \lambda \to \mcal{P}_\kappa \lambda$ as follows: $f_0 (z) := \{ \pi (z) \}$ if $\pi (z) \in \lambda$, and $f_0 (z) := \emptyset$ otherwise. Then we have the following.
\begin{itemize}
\item[(i)] If $x \in S$, and $\mcal{P}_{x \cap \kappa} x$ is closed under $f_0$, then $\pi^* (x) \cap \lambda \subseteq x$.
\end{itemize}
Take $f_1 : \mcal{P}_\kappa \lambda \to \mcal{P}_\kappa \lambda$ such that $f_1 ( \{ \alpha \} ) = \pi^{-1} ( \alpha )$. Then we have the following.
\begin{itemize}
\item[(ii)] If $x \in S$, and $\mcal{P}_{x \cap \kappa} x$ is closed under $f_1$, then $\pi^* (x) \cap \lambda \supseteq x$.
\end{itemize}
Finally, define $f_2 : \mcal{P}_\kappa \lambda \to \mcal{P}_\kappa \lambda$ by
\[
f_2 (z) := \bigcup \pi^{-1} [ h( \pi [ \mcal{P} (z) ] ) ] \, .
\]
We claim the following.
\begin{itemize}
\item[(iii)] If $x \in S$, and $\mcal{P}_{x \cap \kappa} x$ is closed under $f_2$, then $\mcal{P}_{x \cap \kappa} \pi^* (x)$ is closed under $h$.
\end{itemize}

Suppose $x \in S$ and $\mcal{P}_{x \cap \kappa} x$ is closed under $f_2$. Take an arbitrary $w \in \mcal{P}_{x \cap \kappa} \pi^* (x)$. We must show that $h(w) \in \mcal{P}_{x \cap \kappa} \pi^* (x)$.

Let $z := \bigcup \pi^{-1} [w] \in \mcal{P}_{x \cap \kappa} x$. Then $\pi^{-1} [w] \subseteq \mcal{P} (z)$, and so $w \subseteq \pi [ \mcal{P} (z) ]$. Since $h$ is $\subseteq$-increasing, we have $h(w) \subseteq h( \pi [ \mcal{P} (z) ] ) =: v$. Then $f_2 (z) = \bigcup \pi^{-1} [v]$, and so $\pi^{-1} [v] \subseteq \mcal{P} ( f_2 (z) )$, which in turn implies that $v \subseteq \pi [ \mcal{P} ( f_2 (z) ) ]$. Thus $h(w) \subseteq \pi [ \mcal{P} ( f_2 (z) ) ]$.

But $f_2 (z) \in \mcal{P}_{x \cap \kappa} x$ since $\mcal{P}_{x \cap \kappa} x$ is closed under $f_2$. Then $\pi [ \mcal{P} ( f_2 (z) ) ] \subseteq \pi^* (x)$, and $| \pi [ \mcal{P} ( f_2 (z) ) ] | < x \cap \kappa$ since $x \cap \kappa$ is inaccessible. So $\pi [ \mcal{P} ( f_2 (z) ) ] \in \mcal{P}_{x \cap \kappa} \pi^* (x)$. Then $h(w) \in \mcal{P}_{x \cap \kappa} \pi^* (x)$ since $h(w) \subseteq \pi [ \mcal{P} ( f_2 (z) ) ]$.

We have proved (iii). Now, define $f : \mcal{P}_\kappa \lambda \to \mcal{P}_\kappa \lambda$ by $f(z) := f_0 (z) \cup f_1 (z) \cup f_2 (z)$. Since $S$ is strongly stationary in $\mcal{P}_\kappa \lambda$, we can take $x \in S$ with $\mcal{P}_{x \cap \kappa} x$ closed under $f$. Note that $\mcal{P}_{x \cap \kappa} x$ is closed under all $f_0 , f_1 , f_2$. By (i)--(iii), $x$ is as desired.
\end{proof}

Using the previous two lemmas, we can now prove the upward direction of Menas’ theorem for $1$-stationary sets, without any additional assumptions on $\lambda$ and $\lambda'$, if $\kappa$ is Mahlo.

\begin{prop} \label{upwards_Menas_1stat}
Let $\kappa$ be a Mahlo cardinal and let $\lambda , \lambda '$ be cardinals with $\kappa \leq \lambda < \lambda '$. Suppose $S \subseteq \mcal{P}_\kappa \lambda$. If $S$ is $1$-stationary in $\mcal{P}_\kappa \lambda$, then $S \uparrow \lambda '$ is $1$-stationary in $\mcal{P}_\kappa \lambda '$.
\end{prop}

\begin{proof}
Suppose $S$ is $1$-stationary in $\mcal{P}_\kappa \lambda$. Let $\nu := \lambda^{< \kappa}$. By Lemma \ref{lift_lambda_lambda<kappa}, $S \uparrow \nu$ is $1$-stationary in $\mcal{P}_\kappa \nu$. If $\lambda ' \leq \nu$, then $S \uparrow \lambda '$ is $1$-stationary in $\mcal{P}_\kappa \lambda '$ by Proposition \ref{mrtl} and the fact that $( S \uparrow \nu ) \downarrow \lambda' = S \uparrow \lambda'$. If $\lambda ' > \nu$, then $S \uparrow \lambda ' = ( S \uparrow \nu ) \uparrow \lambda '$ is $1$-stationary in $\mcal{P}_\kappa \lambda '$ by Lemma \ref{mltrr} and the fact that $\nu^{< \kappa} = \nu$.
\end{proof}

Finally, as an immediate consequence of Propositions \ref{mrtl} and \ref{upwards_Menas_1stat}, we obtain the announced Menas' theorem for  $1$-stationary sets.

\begin{theorem} \label{thm:Menas_1stat}
Let $\kappa$ be a Mahlo cardinal and $\lambda , \lambda '$ be cardinals with $\kappa \leq \lambda< \lambda'$. Suppose $S \subseteq \mcal{P}_\kappa \lambda$. Then $S$ is $1$-stationary in $\mcal{P}_{\kappa}\lambda$ if and only if $S \uparrow \lambda '$ is $1$-stationary in $\mcal{P}_{\kappa}{\lambda'}$.
\end{theorem}

As we mentioned before, we do not know whether it is necessary to assume $\kappa$ Mahlo in Theorem \ref{thm:Menas_1stat}. This assumption comes from Lemma \ref{lift_lambda_lambda<kappa} to prove the upward direction.

\begin{question}
Is it necessary to assume $\kappa$ Mahlo in Lemma \ref{lift_lambda_lambda<kappa}?
\end{question}

%--------------------------------------------------------------------------------------------------------------------------------------------------------------------------------------------------------------------------------------------------------
\subsection{Menas' Theorem for $2$-stationary sets} \label{sec:Means_2stat}
%--------------------------------------------------------------------------------------------------------------------------------------------------------------------------------------------------------------------------------------------------------

Here we discuss Menas' theorem for $2$-stationary sets. We begin with the downward direction, which follows easily from the case of $1$-stationary sets.

\begin{prop} \label{prop:downward_Menas_2stat}
Let $\kappa$ be a Mahlo cardinal and $\lambda , \lambda '$ be cardinals with $\kappa \leq \lambda < \lambda'$. Suppose $S \subseteq \mcal{P}_\kappa \lambda$. If $S \uparrow \lambda '$ is $2$-stationary in $\mcal{P}_{\kappa}\lambda'$, then $S$ is $2$-stationary in $\mcal{P}_{\kappa}{\lambda}$.
\end{prop}

%----------------------------------------------------------------------------------
\begin{proof}  %---------------------------------------------------------------
Suppose $S \uparrow \lambda '$ is $2$-stationary. Take an arbitrary $1$-stationary $T \subseteq \mcal{P}_{\kappa}{\lambda}$. We find $x \in S$ such that $x \cap \kappa$ is weakly inaccessible, and $T \cap \mcal{P}_{x \cap \kappa} x$ is $1$-stationary.

By Proposition \ref{upwards_Menas_1stat}, $T' := T \uparrow \lambda '$ is $1$-stationary in $\mcal{P}_{\kappa}{\lambda'}$. Since $S \uparrow \lambda '$ is $2$-stationary, there is $x' \in S \uparrow \lambda'$ such that $\mu := x' \cap \kappa$ is weakly inaccessible, and $T' \cap \mcal{P}_\mu x'$ is $1$-stationary. Note that $\mu$ is weakly Mahlo.

Let $x := x' \cap \lambda \in S$. Then $x \cap \kappa = \mu$ is weakly Mahlo, and $( T \cap \mcal{P}_\mu x ) \uparrow x' = T' \cap \mcal{P}_\mu x'$. So, by the same argument as the proof of Proposition \ref{mrtl}, $T \cap \mcal{P}_\mu x$ is $1$-stationary. So $x$ is as desired.
\end{proof}  %---------------------------------------------------------------

%----------------------------------------------------------------------------------

We next show that the upward direction of Menas' theorem can fail for $2$-stationary sets. More precisely, we will construct a model in which $\mcal{P}_\kappa \lambda$ is $2$-stationary, while $\mcal{P}_\kappa \lambda^+$ is not $2$-stationary for some $\kappa$ and $\lambda$. Note that $\mcal{P}_\kappa \lambda \uparrow \lambda^+ = \mcal{P}_\kappa \lambda^+$, so the upward Menas' theorem for $2$-stationary sets would fail in the model.

To construct such a model, we use the following forcing poset, which makes $\mcal{P}_\kappa \nu$ non-$2$-stationary. For regular uncountable cardinals $\kappa$ and $\nu$ with $\kappa \leq \nu$, let $\mbb{Q} ( \kappa , \nu )$ be the poset of all $p$ such that for some $\delta_p < \nu$,
\begin{itemize}
\item[(i)] $p : \mcal{P}_\kappa \delta_p \to 2$,
\item[(ii)] $p^{-1} (1) \cap \mcal{P}_{x \cap \kappa} x$ is non-stationary in $\mcal{P}_{x \cap \kappa} x$ for any $x \in \mcal{P}_\kappa \nu$ with $x \cap \kappa$ regular uncountable.
\end{itemize}
$\mbb{Q} ( \kappa , \nu )$ is ordered by reverse inclusion.

First we prove the following.

\begin{lemma} \label{lem:Q_basics}
Let $\kappa$ and $\nu$ be regular uncountable cardinals with $\kappa \leq \nu$.
\begin{enumerate}
\item For any $\gamma < \nu$ the set $D_\gamma = \{ p \in \mbb{Q} ( \kappa , \nu ) : \gamma \leq \delta_p \}$ is dense in $\mbb{Q} ( \kappa , \nu )$.
\item $\mbb{Q} ( \kappa , \nu )$ is $\nu$-strategically closed.
\end{enumerate}
\end{lemma}

\begin{proof}
Let $\mbb{Q} := \mbb{Q} ( \kappa , \nu )$.

\medskip
\noindent
(1) Suppose $\gamma < \nu$ and $p \in \mbb{Q}$. Take $\delta$ with $\gamma , \delta_p \leq \delta < \nu$, and let $q : \mcal{P}_\kappa \delta \to 2$ be such that $q \supseteq p$, and $q(z) = 0$ for all $z \in \mcal{P}_\kappa \delta \setminus \dom (p)$. It is easy to check that $q \in \mbb{Q}$. Then $q \leq p$ and $q \in D_\gamma$.

\medskip
\noindent
(2) First we make some preliminaries.

Suppose $\alpha$ is a limit ordinal $< \nu$, and $\vec{p} = \langle p_\beta : \beta < \alpha \rangle$ is a descending sequence in $\mbb{Q}$. Then let $\delta_{\vec{p}} := \sup_{\beta < \alpha} \delta_{p_\beta}$, $q_{\vec{p}} := \bigcup_{\beta < \alpha} p_\beta$ and $r_{\vec{p}} : \mcal{P}_\kappa \delta_{\vec{p}} \to 2$ be such that $r_{\vec{p}} \supseteq q_{\vec{p}}$ and $r_{\vec{p}} (z) = 0$ for all $z \in \mcal{P}_\kappa \delta_{\vec{p}} \setminus \dom ( q_{\vec{p}} )$.

Note that if $\vec{p}$ is not eventually constant (i.e.~$\delta_{p_\beta} < \delta_{\vec{p}}$ for all $\beta < \alpha$), and $\cof ( \delta_{\vec{p}} ) < \kappa$, then $\mcal{P}_\kappa \delta_{\vec{p}} \setminus \dom ( q_{\vec{p}} )$ is the set of all $z \in \mcal{P}_\kappa \delta_{\vec{p}}$ with $\sup z = \delta_{\vec{p}}$. Otherwise, $\mcal{P}_\kappa \delta_{\vec{p}} \setminus \dom ( q_{\vec{p}} ) = \emptyset$, and so $r_{\vec{p}} = q_{\vec{p}}$.
Note also that if $\vec{p}$ has a lower bound in $\mbb{Q}$, then $r_{\vec{p}} \in \mbb{Q}$. This is because ${r_{\vec{p}}}^{-1} (1) \subseteq p^{-1} (1)$ for any lower bound $p$ of $\vec{p}$.

Now we give a strategy of Even in the game $\Game_\nu ( \mbb{Q} )$. Suppose $\alpha$ is an even ordinal $< \nu$, and $\vec{p} = \langle p_\beta : \beta < \alpha \rangle$ has been played before the $\alpha$-th stage. If $\alpha$ is a successor ordinal, then Even chooses $p_\alpha$ so that $p_\alpha \lneq p_{\alpha -1}$. If $\alpha$ is a limit ordinal, and $\vec{p}$ has a lower bound in $\mbb{Q}$, then Even chooses $p_\alpha = r_{\vec{p}}$.

We show that the above is a winning strategy of Even. Suppose $\alpha$ is a limit ordinal $< \nu$, and $\vec{p} = \langle p_\beta : \beta < \alpha \rangle$ is a play of $\Game_\nu ( \mbb{Q} )$ in which Even has moved according to the above strategy. We must show that $\vec{p}$ has a lower bound in $\mbb{Q}$. For this it suffices to prove that $r_{\vec{p}} \in \mbb{Q}$.

Before starting, note the following by the strategy of Even at successor stages.
\begin{itemize}
\item[(i)] $p_\gamma \lneq p_{\gamma -1}$, and so $\delta_{p_\gamma} > \delta_{p_{\gamma -1}}$, for all even successor $\gamma < \alpha$.
\end{itemize}
Below, let $\delta$, $q$ and $r$ denote $\delta_{\vec{p}}$, $q_{\vec{p}}$ and $r_{\vec{p}}$, respectively.

Suppose $x \in \mcal{P}_\kappa \nu$, and $\mu := x \cap \kappa$ is a regular uncountable cardinal. We show that $r^{-1} (1) \cap \mcal{P}_\mu x$ is non-stationary. This is clear if $x \not\subseteq \delta$. So assume $x \subseteq \delta$. We consider several cases. Note that $\delta$ is a limit ordinal by (i).

First suppose $\sup x < \delta$. Then there is $\beta < \alpha$ with $x \in \mcal{P}_\kappa \delta_{p_\beta}$. Then $r^{-1} (1) \cap \mcal{P}_\mu x = {p_\beta}^{-1} (1) \cap \mcal{P}_\mu x$, and it is non-stationary since $p_\beta \in \mbb{Q}$.

Next suppose $\sup x = \delta$, and $\cof ( \delta ) < \mu$. Then there are club many $z \in \mcal{P}_\mu x$ with $\sup z = \delta$, and $r(z) = 0$ for all such $z$. So $r^{-1} (1) \cap \mcal{P}_\mu x$ is non-stationary.

Finally suppose $\sup x = \delta$, and $\cof ( \delta ) \geq \mu$. Let $C$ be the collection of $\delta_{p_\beta}$ for all limit $\beta < \alpha$. Then $C$ is club in $\delta$. So the set
\[
Z := \{ z \in \mcal{P}_\mu x : \sup z \notin z \wedge \sup z \in C \}
\]
is club in $\mcal{P}_\mu x$. But $r(z) = 0$ for all $z \in Z$, by (i) and the strategy of Even at limit stages. So $r^{-1} (1) \cap \mcal{P}_\mu x$ is non-stationary.
\end{proof}

By Lemma \ref{lem:Q_basics}, a forcing extension by $\mbb{Q} ( \kappa , \nu )$ adds no new sequences of ordinals of length $< \nu$. So it preserves regularities of $\kappa$ and $\nu$ and does not change $\mcal{P}_\kappa \nu$.

Now we prove that $\mbb{Q} ( \kappa , \nu )$ makes $\mcal{P}_\kappa \nu$ non-$2$-stationary.

\begin{lemma} \label{lem:Q_2stat}
Let $\kappa$ be a weakly Mahlo cardinal and let $\nu$ be a regular cardinal $\geq \kappa$. Then $\mbb{Q} ( \kappa , \nu )$ forces that $\mcal{P}_\kappa \nu$ is not $2$-stationary.
\end{lemma}

\begin{proof}
Let $\mbb{Q} := \mbb{Q} ( \kappa , \nu )$. Suppose $G$ is a $\mbb{Q}$-generic filter over $V$. In $V[G]$, let $S := ( \bigcup G )^{-1} (1)$.
Then, by the definition of $\mbb{Q}$ and Fact \ref{fact:1-stat} (1), in $V[G]$, $S \cap \mcal{P}_{x \cap \kappa} x$ is not $1$-stationary for any $x \in \mcal{P}_\kappa \nu$ with $x \cap \kappa$ weakly inaccessible. Thus it suffices to prove that $S$ is $1$-stationary in $\mcal{P}_\kappa \nu$.

We work in $V$. Suppose $\dot{T}$ is a $\mbb{Q}$-name for an unbounded subset of $\mcal{P}_\kappa \nu$, and $p \in \mbb{Q}$. It suffices to find $q \leq p$ and $x \in \dom (q)$ such that $x \cap \kappa$ is weakly inaccessible, $q(x) = 1$, and there are unboundedly many $z \in \mcal{P}_{x \cap \kappa} x$ such that $q \Vdash ``\, z \in \dot{T} \,"$.

Note that $\mbb{Q}$ is $\kappa$-strategically closed by Lemma \ref{lem:Q_basics}. Then we can recursively construct a descending sequence $\langle p_\alpha : \alpha < \kappa \rangle$ in $\mbb{Q}$ below $p$ and a $\subseteq$-increasing sequence $\langle z_\alpha : \alpha < \kappa \rangle$ in $\mcal{P}_\kappa \nu$ so that for all $\alpha < \kappa$,
\begin{itemize}
\item[(i)] $\alpha \subseteq z_\alpha$,
\item[(ii)] $z_\alpha \subseteq \delta_{p_\alpha}$, but $z_{\alpha + 1} \not\subseteq \delta_{p_\alpha}$,
\item[(iii)] $p_\alpha \Vdash ``\, z_\alpha \in \dot{T} \,"$.
\end{itemize}
Note that $\langle \delta_{p_\alpha} : \alpha < \kappa \rangle$ is strictly increasing by (ii).

Since $\kappa$ is weakly Mahlo, we can take a weakly inaccessible $\mu < \kappa$ such that $| z_\alpha | < \mu$ for all $\alpha < \mu$. Let $x := \bigcup_{\alpha < \mu} z_\alpha$ and $\delta := \sup_{\alpha < \mu} \delta_{p_\alpha}$. Then $x \cap \kappa = \mu$ by (i). Moreover, by (ii), $x \in \mcal{P}_\kappa \delta$, and $\sup x = \delta$.

Let $q : \mcal{P}_\kappa \delta \to 2$ be such that $q \supseteq \bigcup_{\alpha < \mu} p_\alpha$, $q(x) = 1$, and $q (z) = 0$ for any $z \in \mcal{P}_\kappa \delta$ with $\sup (z) = \delta$ and $z \neq x$. Then, using the fact that $\langle p_\alpha : \alpha < \mu \rangle$ has a lower bound $p_\mu$ in $\mbb{Q}$, it is easy to check that $q \in \mbb{Q}$.

$q(x) = 1$ by  definition. Note that $q$ is a lower bound of $\langle p_\alpha : \alpha < \mu \rangle$. So $q \leq p$. Moreover $q \Vdash ``\,z_\alpha \in \dot{T}\,"$ for all $\alpha < \mu$ by (iii). But $\{ z_\alpha : \alpha < \mu \}$ is unbounded in $\mcal{P}_\mu x$ by the regularity of $\mu$ and the fact that $x = \bigcup_{\alpha < \mu} z_\alpha$. Thus $q$ and $x$ are as desired.
\end{proof}

Now we prove the following, which shows the consistency of the failure of Menas' theorem for $2$-stationary sets.

\begin{prop} \label{fum2stp}
Suppose $\kappa$ is a weakly Mahlo cardinal $\lambda$ is a cardinal $\geq \kappa$ with $\lambda^{< \kappa} = \lambda$, and $\mcal{P}_\kappa \lambda$ is $2$-stationary. Then there is a $\lambda^+$-strategically closed forcing extension in which $\mcal{P}_\kappa \lambda$ is $2$-stationary but $\mcal{P}_\kappa \lambda^+$ is not $2$-stationary
\end{prop}

%----------------------------------------------------------------------------------
\begin{proof}
By Lemma \ref{lem:Q_basics} (2), $\mbb{Q} ( \kappa , \lambda^+ )$ is $\lambda^+$-strategically closed. Suppose $G$ is a $\mbb{Q} ( \kappa , \lambda^+ )$-generic filter over $V$. Then $\mcal{P}_\kappa \lambda^+$ is not $2$-stationary by Lemma \ref{lem:Q_2stat}. Note that $\mcal{P} ( \mcal{P}_\kappa \lambda )$ is absolute between $V$ and $V[G]$ by the $\lambda^+$-strategically closure of $\mbb{Q} ( \kappa , \lambda^+ )$ and the assumption that $\lambda^{< \kappa} = \lambda$. So $\mcal{P}_\kappa \lambda$ is $2$-stationary in $V[G]$.
\end{proof}  %---------------------------------------------------------------

We end this seciton with two questions. First, we do not know if (upwards) Menas' Theorem for $2$-stationary sets is consistent.

\begin{question}
Is (upwards) Menas' Theorem for $2$-stationary sets consistent in non-trivail sense? More precisely, is it consistent with ZFC that there is a weakly Mahlo cardinal $\kappa$ with the following properties?
\begin{itemize}
\item[(i)] $\mcal{P}_\kappa \kappa$ is $2$-stationary,
\item[(ii)] For any cardinals $\lambda , \lambda '$ with $\kappa \leq \lambda < \lambda '$ and any $S \subseteq \mcal{P}_\kappa \lambda$, $S$ is $2$-stationary in $\mcal{P}_\kappa \lambda$ if and only if $S \uparrow \lambda '$ is $2$-stationary in $\mcal{P}_\kappa \lambda '$.
\end{itemize}
\end{question}

We present one more question. In Proposition \ref{fum2stp}, we assumed that $\lambda$ is a cardinal with $\lambda^{< \kappa} = \lambda$. Note that if $\lambda$ is a singular cardinal of cofinality $< \kappa$, then this assumption fails. We do not know the answer of the following question.

\begin{question}
Suppose $\kappa$ is a (weakly) Mahlo cardinal, $\lambda$ is a singular cardinal $> \kappa$ of cofinality $< \kappa$, and $\mcal{P}_\kappa \lambda$ is $2$-stationary. Then, is $\mcal{P}_\kappa \lambda^+$ $2$-stationary?
\end{question}

%--------------------------------------------------------------------------------------------------------------------------------------------------------------------------------------------------------------------------------------------------------
\section{$2$-stationarity and strong compactness} \label{sec:2stat_stcmpct}
%--------------------------------------------------------------------------------------------------------------------------------------------------------------------------------------------------------------------------------------------------------

In this section we prove the main theorem. Our proof is based on Krueger \cite{Kr} and Sakai \cite{S}, both of whom investigated the difference between supercompactness and strong compactness.

\begin{mthm}
Assume GCH holds, $\kappa$ is a supercommpact cardinal, and there is no measurable cardinal $> \kappa$. Then there is a forcing extension in which
\begin{itemize}
\item[(I)] $\kappa$ is strongly compact,
\item[(II)] $\mcal{P}_\kappa \lambda$ is not $2$-stationary in itself for any cardinal $\lambda > \kappa$.
\end{itemize}
\end{mthm}

We first describe the strategy of the proof. By downward Menas' Theorem for $2$-stationary sets (Proposition \ref{prop:downward_Menas_2stat}), it is enough to construct a forcing extension in which $\kappa$ is strongly compact and $\mcal{P}_\kappa \kappa^+$ is not $2$-stationary. In our forcing extension, the set
\[
S( \kappa , \kappa^+ ) := \{ x \in \mcal{P}_\kappa \kappa^+ \mid x \cap \kappa \in \kappa \;\&\; \otp (x) = ( x \cap \kappa )^+ \}
\]
will be a counterexample of $2$-stationarity of $\mcal{P}_\kappa \kappa^+$, i.e.~a non-reflecting $1$-stationary subset of $\mcal{P}_\kappa \kappa^+$. Recall that if $\kappa$ is $\kappa^+$-supercompact, then $S( \kappa , \kappa^+ )$ belongs to any normal $\kappa$-ultrafilter over $\mcal{P}_\kappa \kappa^+$, and so it is $1$-stationary.

Our forcing is designed to make $S( \kappa , \kappa^+ )$ non-reflecting while preserving the $1$-stationarity of $S( \kappa , \kappa^+ )$ and the strong compactness of $\kappa$. As in \cite{Kr, S}, the construction consists of two main steps. The first step is a forcing, due to Apter \& Cummings \cite{AC} that adds partial square sequence at $\kappa$ while preserving the supercompactness of $\kappa$. This ensures that $x \cap \kappa$ is measurable for all but non-stationary many $x \in S( \kappa , \kappa^+ )$. The second step is an Easton support iteration of a forcing due to Shioya \cite{Sh}, which makes $S( \alpha , \alpha^+ )$ non-stationary for all measurable $\alpha < \kappa$.

It is then easy to see that $S( \kappa , \kappa^+ )$ is non-reflecting after these two steps. In addition, we can prove that $S( \kappa , \kappa^+ )$ remains $1$-stationary by a somewhat standard argument. Moreover, we can show that $\kappa$ remains strongly compact using Magidor's argument, as presented in Cummings \cite{C} \S 22.

The remainder of this section is constructed as follows. In \S \ref{sec:S_psquare}, we review basic properties of $S( \kappa , \kappa^+ )$ and the partial square principle. In \S \ref{sec:pres_stcmpct} we review Magidor's argument for the preservation of strong compactness by certain Easton support iterations. Finally, in \S \ref{sec:main} we prove the main theorem.

%--------------------------------------------------------------------------------------------------------------------------------------------------------------------------------------------------------------------------------------------------------
\subsection{$S(\kappa, \kappa^{+})$ and partial square} \label{sec:S_psquare}
%--------------------------------------------------------------------------------------------------------------------------------------------------------------------------------------------------------------------------------------------------------

The set $S( \kappa , \kappa^+ )$ plays an important role in the study of supercompact cardinals and has been well studied. Although many of the facts below, or closely related variants of them, can be found in the literature (see, for example, \cite{K,Je,Kr1,Kr2,S}), we will include most of the proofs for completeness of this paper. We also consider the sets $S_< ( \kappa , \kappa^+ )$ and $S_> ( \kappa , \kappa^+ )$.

\begin{definition*}
For a weakly inaccessible cardinal $\kappa$, let
\begin{eqnarray*}
S ( \kappa , \kappa^+ ) & := & \{ x \in \mcal{P}_\kappa \kappa^+ : x \cap \kappa \in \kappa \;\&\; \otp (x) = ( x \cap \kappa )^+ \} , \\
S_<(\kappa, \kappa^+ ) & := & \{ x \in \mcal{P}_\kappa \kappa^+ :  x \cap \kappa \in \kappa \;\&\; \otp (x) < (x\cap \kappa)^+ \} , \\
S_> (\kappa, \kappa^+ ) & := & \{ x \in \mcal{P}_\kappa \kappa^+ :  x \cap \kappa \in \kappa \;\&\; \otp (x) > (x\cap \kappa)^+ \} .
\end{eqnarray*}
\end{definition*}

We first note that, for the latter two sets, stationarity is completely determined.

\begin{lemma}\label{lem:S_<_S_>_stat} 
Let $\kappa$ be a weakly inaccessible cardinal. Then $S_< ( \kappa , \kappa^+ )$ is stationary in $\mcal{P}_\kappa \kappa^+$, and $S_> ( \kappa , \kappa^+ )$ is non-stationary in $\mcal{P}_\kappa \kappa^+$.
\end{lemma}

%----------------------------------------------------------------------------------

\begin{proof} %----------------------------------- 
The stationarity of $S_< ( \kappa , \kappa^+ )$ is standard. See \cite{K} Proposition 25.5. Here we prove that $S_> ( \kappa , \kappa^+ )$ is non-stationary.

For each $\alpha < \kappa^+$ take an injection $\pi_\alpha : \alpha \to \kappa$, and let $Z$ be the set of all $x \in \mcal{P}_\kappa \kappa^+$ such that $\pi_\alpha [ x \cap \alpha ] \subseteq x \cap \kappa$ for any $\alpha \in x$. Then it is easy to check that $Z$ is club in $\mcal{P}_\kappa \kappa^+$. We show that $S_> ( \kappa , \kappa^+ ) \cap Z = \emptyset$.

Suppose $x \in S_> ( \kappa , \kappa^+ )$. Let $\alpha$ be the $( x \cap \kappa )^+$-th element of $x$. Then $\pi_\alpha [ x \cap \alpha ] \not\subseteq x \cap \kappa$ since $| x \cap \alpha | = ( x \cap \kappa )^+$, and $\pi_\alpha$ is injective. So $x \notin Z$.
\end{proof}%-------------------------------------

On the other hand, $S( \kappa , \kappa^+ )$ can be either stationary or non-stationary. If $\kappa$ is $\kappa^+$-supercompact, then we have the following.

%----------------------------------------------------------------------------------

\begin{lemma} \label{lem:S_stat_spcmpct} 
Suppose $\kappa$ is a $\kappa^+$-supercompact cardinal. Then $S(\kappa, \kappa^{+})$ is $n$-stationary in $\mcal{P}_{\kappa}{\kappa^+}$ for any $n < \omega$.
\end{lemma}

%----------------------------------------------------------------------------------

\begin{proof} %-----------------------------------
Let $U$ be a normal $\kappa$-ultrafilter over $\mcal{P}_\kappa \kappa^+$ and $j : V \to \mrm{Ult} ( V , U )$ be the ultrapower map. Then $j[ \kappa^+ ] \in j( S( \kappa , \kappa^+ ) )$, and so $S( \kappa , \kappa^+ ) \in U$. Then, by Proposition \ref{prop:spcmpct_n_stat}, $S( \kappa , \kappa^+ )$ is $n$-stationary for all $n < \omega$.
\end{proof} %-------------------------------------------------------

Baumgartner proved that if $0^\sharp$ does not exist, then $S( \kappa , \kappa^+ )$ is non-stationary for every regular uncountable $\kappa$; see \cite{K} Theorem 25.7. Similarly, Shioya proved that $S( \kappa , \kappa^+ )$ can be made non-stationary by a forcing with very useful properties. We will use Shioya's forcing in the proof of the main theorem.

\begin{fact}[Shioya \cite{Sh}]\label{fact:Shioya}
Let $\kappa$ be an inaccessible cardinal. Then there is a $\kappa$-strategically closed $\kappa^+$-c.c.~forcing poset $\mbb{S}_\kappa$ of size $\kappa^+$ which forces $S( \kappa , \kappa^+ )$ to be non-stationary in $\mcal{P}_\kappa \kappa^+$.
\end{fact}

It is well-known that if $\square_\kappa$ holds, then $S( \kappa , \kappa^+ )$ is non-stationary. Following the ideas in \cite{Kr,S}, we will use partial version of this fact.

\begin{definition*}
Suppose $\kappa$ is an uncountable cardinal, and $E \subseteq \ltpt ( \kappa^+ )$. Then $\square_{\kappa}^E$ is the statement that there exists a sequence $\langle c_{\alpha} : \alpha \in E \rangle$ satisfying the following for all $\alpha , \alpha ' \in E$.
\begin{itemize}
\item[(i)] $c_{\alpha}$ is a club subset of $\alpha$ of order-type $\leq \kappa$.
\item[(ii)] $c_{\alpha} \cap \beta = c_{\alpha'} \cap\beta$ for any $\beta \in ( \ltpt (c_{\alpha}) \cup \{\alpha\} ) \cap ( \ltpt (c_{\alpha'}) \cup \{\alpha'\} )$.
\end{itemize}
\end{definition*}

Before discussing our use of $\square_{\kappa}^E$, we recall a useful fact concerning $\mcal{P}_\kappa \kappa^+$.

\begin{lemma}\label{lem:gamma_closed}
Suppose $\kappa$ is a regular uncountable cardinal. Then there are club many $x \in \mcal{P}_\kappa \kappa^+$ such that
\begin{itemize}
\item[(i)] $x \cap \kappa \in \kappa$,
\item[(ii)] $x$ is $\gamma$-closed for any regular $\gamma \in x \cap \kappa$ with $\gamma \neq \cof ( x \cap \kappa ) , \cof ( \sup x )$.
\end{itemize}
\end{lemma} 

\begin{proof}
Let $\theta$ be a sufficiently large regular cardinal. It suffices to prove that if $M \prec \langle H_\theta , \in , \kappa \rangle$, $|M| < \kappa$, and $M \cap \kappa \in \kappa$, then $x = M \cap \kappa^+$ satisfies (ii).

Suppose $M \prec \langle H_\theta , \in , \kappa \rangle$, $|M| < \kappa$, and $\mu := M \cap \kappa \in \kappa$, and let $x := M \cap \kappa^+$. Suppose $\gamma$ is a regular cardinal $< \mu$ with $\gamma \neq \cof ( \mu ) , \cof ( \sup x )$. We show that $x$ is $\gamma$-closed. For a contradiction, assume $y \subseteq x$, $\otp (y) = \gamma$, and $\sup (y) \notin x$.

First note that $y$ is bounded in $x$ since $\otp (y) = \gamma \neq \cof ( \sup x )$. Let $\alpha := \min ( x \setminus \sup (y) )$. Then
\[
\sup ( M \cap \alpha ) = \sup ( x \cap \alpha ) = \sup (y) < \alpha \in x = M \cap \kappa^+ .
\]
Note that $\alpha$ is a limit ordinal by the elementarity of $M$. Note also that $\cof ( \alpha ) \in M$. The rest of the proof splits into two cases.

First suppose $\cof ( \alpha ) < \kappa$. Then $\cof ( \alpha ) \in M \cap \kappa$, and so $\cof ( \alpha ) \subseteq M$. Take a cofinal subset $A \in M$ of $\alpha$ with $\otp (A) = \cof ( \alpha )$. Then $A \subseteq M$ since $\cof ( \alpha ) \subseteq M$. Then $\sup ( M \cap \alpha ) \geq \sup (A) = \alpha$. This contradicts that $\sup ( M \cap \alpha ) < \alpha$.

Next suppose $\cof ( \alpha ) = \kappa$. Take a sequence $\langle \alpha_\xi \mid \xi < \kappa \rangle \in M$ which is increasing continuous and cofinal in $\alpha$. Then it easily follows from the elementarity of $M$ that $\sup ( M \cap \alpha ) = \alpha_{M \cap \kappa} = \alpha_\mu$. Then
\[
\gamma = \cof ( \sup (y) ) = \cof ( \sup ( M \cap \alpha ) ) = \cof ( \mu ) .
\]
This contradicts that $\gamma \neq \cof ( \mu )$.
\end{proof} %----------------------------

We now state one important property of $\square_{\kappa}^E$, that will be relevant for this paper. This was essentially proved in \cite{Kr}, but using Lemma \ref{lem:gamma_closed}, we can provide an easy proof for completeness.

\begin{lemma} \label{lem:psquare_S} 
Let $\kappa$ be a weakly inaccessible cardinal. Suppose $E \subseteq \ltpt ( \kappa^+ )$, and $\square_\kappa^E$ holds. Then the set $\{ x \in S( \kappa , \kappa^+ ) \mid \sup x \in E \}$ is non-stationary in $\mcal{P}_\kappa \kappa^+$.
\end{lemma}

%----------------------------------------------------------------------------------

\begin{proof}
Let $\langle c_\alpha : \alpha \in E \rangle$ be a witness of $\square_\kappa^E$. Let $D := E \cup \bigcup_{\alpha \in E} \ltpt ( c_\alpha )$. For each $\beta \in D \setminus E$, taking $\alpha \in E$ with $\beta \in \ltpt ( c_\alpha )$, let $c_\beta := c_\alpha \cap \beta$. Note that $c_\beta$ is independent of the choice of $\alpha$. Then $\vec{c} := \langle c_\alpha : \alpha \in D \rangle$ satisfies the following.
\begin{itemize}
\item[(i)] For all $\alpha \in D$, $c_\alpha$ is a club subset of $\alpha$ of order-type $\leq \kappa$.
\item[(ii)] If $\alpha \in D$, and $\beta \in \ltpt ( c_\alpha )$, then $\beta \in D$, and $c_\beta = c_\alpha \cap \beta$.
\end{itemize}

Let $\theta$ be a sufficiently large regular cardinal and $\mcal{M}$ be $\langle \mcal{H}_\theta , \in , \kappa , \vec{c} \rangle$. It suffices to prove that if $M \prec \mcal{M}$, $|M| < \kappa$, and $\sup ( M \cap \kappa^+ ) \in E$, then $M \cap \kappa^+ \notin S( \kappa , \kappa^+ )$. On the contrary assume $M \prec \mcal{M}$, $|M| < \kappa$, $\sup ( M \cap \kappa^+ ) \in E$, and $M \cap \kappa^+ \in S( \kappa , \kappa^+ )$.

Note that $\mu := M \cap \kappa$ is an uncountable limit cardinal. Then by Lemma \ref{lem:gamma_closed}, $M \cap \kappa^+$ is $\gamma$-closed for some regular $\gamma < \mu$. So $M \cap \kappa^+$ is stationary in $\alpha := \sup ( M \cap \kappa^+ )$. Here note that $\cof ( \alpha ) = \mu^+$ since $M \cap \kappa^+ \in S( \kappa , \kappa^+ )$. Then we can take $\beta \in \ltpt ( c_\alpha ) \cap M$ with $\otp ( c_\alpha \cap \beta ) > \mu$. But $c_\alpha \cap \beta = c_\beta \in M$, and so $\otp ( c_\alpha \cap \beta ) \in M \cap \kappa = \mu$. This is a contradiction.
\end{proof}

Now, we present a result showing that a  partial square
principle can be forced while preserving the supercompactness of $\kappa$. This forcing will be used as the first step in the proof of the main theorem.

\begin{fact}[Apter \& Cummings \cite{AC}]\label{fact:Apter_Cummings}
Assume that $\kappa$ is a supercompact cardinal. Then there is a forcing poset of size $\kappa^+$ which forces the following.
\begin{itemize}
\item[(i)] $\kappa$ is supercompact.
\item[(ii)] $\square_\kappa^E$ holds for the set
\[
E := \ltpt ( \kappa^+ ) \setminus \bigcup \{ E^{\kappa^+}_{\alpha^+} : \mbox{$\alpha$ is a measurable cardinal $< \kappa$} \} .
\]
\end{itemize}
Moreover, if GCH holds, then $\mbb{P}$ forces GCH.
\end{fact}

We end this section by showing that $S( \kappa , \kappa^+ )$ does not reflect to any point in $S_< ( \kappa , \kappa^+ )$.

\begin{lemma} \label{lem:S_nonrefl_S<}
Let $\kappa$ be a weakly inaccessible cardinal. Then $S(\kappa, \kappa^{+})\cap \mcal{P}_{x \cap\kappa}{x}$ is non-stationary in $\mcal{P}_{x \cap\kappa} x$ for any $x \in S_{<}(\kappa, \kappa^{+})$ such that $x \cap \kappa$ is a regular uncountable cardinal.
\end{lemma}

\begin{proof}
Suppose $x \in S_< ( \kappa , \kappa^+ )$ and $x \cap \kappa$ is a regular uncountable cardinal. Then $|x| = x \cap \kappa$. Take a bijection $f : x \cap \kappa \to x$, and let $Z := \{ f[ \alpha ] \cup \alpha : \alpha < x \cap \kappa \}$. Then it is easy to check that $Z$ is club in $\mcal{P}_{x \cap \kappa} x$. Moreover $Z \cap S( \kappa , \kappa^+ ) = \emptyset$ since $|z| = | z \cap \kappa |$ for all $z \in Z$.
\end{proof}

By the above lemma and Fact \ref{fact:1-stat} (1) we have the following corollary.

\begin{corollary}\label{cor:dS_S<}
$d_1(S(\kappa, \kappa^{+})) \cap S_{<}(\kappa, \kappa^{+})= \emptyset$ for any weakly inaccessible $\kappa$.
\end{corollary}

%--------------------------------------------------------------------------------------------------------------------------------------------------------------------------------------------------------------------------------------------------------
\subsection{Preservation of strong compactness} \label{sec:pres_stcmpct}
%--------------------------------------------------------------------------------------------------------------------------------------------------------------------------------------------------------------------------------------------------------

A key ingredient in the proof of the main theorem is the preservation of strong compactness under forcing. In this section, we address this issue.

Magidor \cite{Ma} proved the consistency of the statement that the least measurable cardinal is strongly compact using an iteration of Prikry forcing. Later, he found another proof, which uses an Easton support iteration of shooting non-reflecting stationary sets. This latter proof is presented in Cummings \cite{C} \S 22. The following proposition follows from a slight generalization of that argument.

\begin{prop} \label{prop:pres_strcmpct}
Assume that GCH holds, $\kappa$ is a supercompact cardinal, and there is no measurable cardinal $> \kappa$. Let $( \langle \mbb{P}_\alpha : \alpha \leq \kappa \rangle , \langle \dot{\mbb{Q}}_\alpha : \alpha < \kappa \rangle )$ be an Easton support forcing iteration with the following properties.
\begin{itemize}
\item[(I)] If $\alpha < \kappa$ is not measurable, then $\mbb{P}_\alpha$ forces that $\dot{\mbb{Q}}_\alpha$ is a trivial poset.
\item[(II)] If $\alpha < \kappa$ is measurable, then $\mbb{P}_\alpha$ forces that $\dot{\mbb{Q}}_\alpha$ is an $\alpha$-strategically closed $\alpha^+$-c.c.~forcing poset of size $\leq \alpha^+$.
\end{itemize}
Then $\mbb{P}_\kappa$ forces that $\kappa$ is strongly compact.
\end{prop}

This proposition can be proved by the argument in \cite{C} \S 22, although it is not explicitly stated there. We therefore include the proof for the reader's convenience and for completeness. We assume some familiarity with the basic arguments for lifting elementary embeddings in generic extensions, as presented in \cite{C} \S 9 and \S11.

Before proving the proposition, we recall some preliminary material.

First we briefly review the idea of term forcings, due to Laver. Let $\mbb{P}$ be a forcing poset and let $\dot{\mbb{Q}}$ be a $\mbb{P}$-name for a forcing poset. The \emph{term poset} $\mbb{A} ( \mbb{P} , \dot{\mbb{Q}} )$ is the poset consisting of all least-rank $\mbb{P}$-names $\dot{q}$ such that $\Vdash_\mbb{P} ``\, \dot{q} \in \dot{\mbb{Q}}\,"$. Here a $\mbb{P}$-name $\dot{a}$ is said to be \emph{least-rank} if there is no $\mbb{P}$-name $\dot{b}$ of lower rank such that $\Vdash_\mbb{P} ``\, \dot{a} = \dot{b} \, "$. (In \cite{C}, least-rank names are called canonical names.) For $\dot{q}_0,\dot{q}_1\in\mbb{A} ( \mbb{P} , \dot{\mbb{Q}} )$, we define $\dot{q}_0 \leq \dot{q}_1$ if $\Vdash_\mbb{P} ``\,\dot{q}_0 \leq \dot{q}_1 \, "$.

We shall use the following two lemmas, which correspond to Propositions 22.3 and 22.4 of \cite{C}. For the reader's convenience, we include the proofs.

\begin{lemma}\label{lem:term_basic}
Let $M$ be an inner model. In $M$ let $\mbb{P}$ be a forcing poset and $\dot{\mbb{Q}}$ be a $\mbb{P}$-name for a forcing poset. Suppose that $G$ is a $\mbb{P}$-generic filter over $M$, and $I$ is an $\mbb{A} (\mbb{P},\dot{\mbb{Q}})^M$-generic filter over $M$. Then $H := \{i_G (\dot{q}) : \dot{q} \in I \}$ is an $i_G (\dot{\mbb{Q}})$-generic filter over $M[G]$.
\end{lemma}

\begin{proof}
Let $\mbb{A} := \mbb{A} ( \mbb{P} , \dot{\mbb{Q}} )^M$ and $\mbb{Q} := i_G ( \dot{\mbb{Q}} )$.

First we show that $H$ is directed downward. Let $q_0 , q_1 \in H$. For $i = 0,1$, take $\dot{q}_i \in I$ with $q_i = i_G ( \dot{q}_i )$. Since $I$ is a filter, we can take $\dot{q} \in I$ with $\dot{q} \leq \dot{q}_0 , \dot{q}_1$ in $\mbb{A}$. Then $q := i_G ( \dot{q} ) \in H$, and $q \leq q_0 , q_1$ in $\mbb{Q}$.

Next we show that $H$ is upwards closed. Suppose $q_0 \in H$ and $q_0 \leq q_1$ in $\mbb{Q}$. Take $\dot{q}_0 \in I$ with $q_0 = i_G ( \dot{q}_0 )$. Then we can take $\dot{q}_1 \in \mbb{A}$ with $\dot{q}_0 \leq \dot{q}_1$ in $\mbb{A}$ and $q_1 = i_G ( \dot{q}_1 )$. Then $\dot{q}_1 \in I$ since $I$ is a filter, and so $q_1 \in H$.

Finally we prove the genericity of $H$. Suppose $D$ is a dense subset of $\mbb{Q}$ in $M[G]$. Let $\dot{D}$ be a $\mbb{P}$-name with $D = i_G ( \dot{D} )$. We may assume that $\mbb{P}$ forces $\dot{D}$ to be a dense subset of $\dot{\mbb{Q}}$. Then it is easy to check that the set
\[
D^* := \{ \dot{q} \in \mbb{A} : M \models \; \Vdash_{\mbb{P}} ``\, \dot{q} \in \dot{D} \," \} \in M
\]
is dense in $\mbb{A}$. Take $\dot{q} \in I \cap D^*$ by the genericity of $I$. Then $i_G ( \dot{q} ) \in H \cap D$.
\end{proof}

\begin{lemma} \label{lem:term_basic2}
Let $\kappa$ be a regular uncountable cardinal. Suppose $\mbb{P}$ is a forcing poset, and $\dot{\mbb{Q}}$ is a $\mbb{P}$-name for a $\kappa$-strategically closed forcing poset. Then $\mbb{A} ( \mbb{P} , \dot{\mbb{Q}} )$ is $\kappa$-strategically closed.
\end{lemma}

\begin{proof}
Let $\mbb{A} := \mbb{A} ( \mbb{P} , \dot{\mbb{Q}} )$. Let $\dot{\tau}$ be a $\mbb{P}$-name for Even's winning strategy in $\Game_\kappa ( \dot{\mbb{Q}} )$. Let $\tau^*$ be Even's strategy in $\Game_\kappa ( \mbb{A} )$ such that if $\langle \dot{q}_\beta : \beta < \alpha \rangle$ has been chosen for an even $\alpha < \kappa$, then choose $\dot{q}_\alpha \in \mbb{A}$ with $\Vdash_\mbb{P} ``\,\dot{q}_\alpha = \dot{\tau} ( \langle \dot{q}_\beta : \beta < \alpha \rangle ) \,"$. Then it is easy to check that $\tau^*$ is a winning strategy of Even in $\Game_\kappa ( \mbb{A} )$.
\end{proof}

We next introduce our notation and recall some facts concerning large-cardinal embeddings. We will use the following well-known fact due to Mitchell. The proof can be found in Jech \cite{Je} Chapter 19. An ultrafilter $U$ satisfying the conclusion of next fact is called an ultrafilter of \emph{Mitchell order} $0$.

\begin{fact} \label{fact:Mitchell_order_0}
Suppose $\kappa$ is a measurable cardinal. Then there is a normal $\kappa$-ultrafilter $U$ over $\kappa$ such that $\{ \alpha < \kappa : \mbox{$\alpha$ is not measurable} \} \in U$.
\end{fact}

Let $\kappa$ and $\lambda$ be cardinals with $\kappa \leq \lambda$. Then a $( \kappa , \lambda )$-\emph{strongly compact embedding} is an elementary embedding $j$ from $V$ to some inner model $M$ such that
\begin{itemize}
\item[(i)] $\mrm{crit} (j) = \kappa$, and $j( \kappa ) > \lambda$,
\item[(ii)] there is $x \in \mcal{P}_{j( \kappa )} j( \lambda )^M$ with $j[ \lambda ] \subseteq x$.
\end{itemize}
Recall that $\kappa$ is $\lambda$-strongly compact if and only if there is a $( \kappa , \lambda )$-strongly compact embedding.

Suppose that $M$ and $N$ are inner models, $j : M \to N$ is an elementary embedding and $\kappa$ is an ordinal. We say that $j$ has \emph{width} $\leq \kappa$ if for any $x \in N$ there is a function $f \in M$ with $\dom (f) = \kappa$ and an ordinal $\alpha < j( \kappa )$ such that $x = j(f)( \alpha )$. Note that if $\kappa$ is a measurable cardinal in $M$, and $j : M \to N$ is the ultrapower map by some $\kappa$-ultrafilter over $\kappa$, then $j$ has width $\leq \kappa$. The following lemma corresponds to Proposition 15.1 of \cite{C}. We give the proof for the convenience of the readers.

\begin{lemma} \label{lem:width_generic}
Let $\kappa$ be an ordinal. Suppose $M$ and $N$ are inner models, and $j : M \to N$ is an elementary embedding of width $\leq \kappa$. Let $\mbb{P}$ be a forcing poset in $M$ such that $\mbb{P}$ is $\kappa^+$-strategically closed in $M$, and let $G$ be a $\mbb{P}$-generic filter over $M$. Then $j[G]$ generates a $j( \mbb{P} )$-generic filter over $N$.
\end{lemma}

\begin{proof}
Suppose $D$ is a dense open subset of $j( \mbb{P} )$ in $N$. It suffices to show that $j[G] \cap D \neq \emptyset$. Take a function $f \in M$ on $\kappa$ and $\alpha < j( \kappa )$ such that $D = j(f)( \alpha )$. We may assume $f( \beta )$ is a dense open subset of $\mbb{P}$ for all $\beta < \kappa$. Since $\mbb{P}$ is $\kappa^+$-strategically closed in $M$, $D^* := \bigcap_{\beta < \kappa} f( \beta )$ is dense in $\mbb{P}$. So we can take $p \in G \cap D^*$. Then $j(p) \in j[G] \cap j( D^* )$. But $j( D^* ) = \bigcap_{\beta < j( \kappa )} j(f)( \beta ) \subseteq D$ by the elementarity of $j$. So $j(p) \in j[G] \cap D$.
\end{proof}

Let $M$ and $N$ be inner models, $j : M \to N$ be an elementary embedding and $\mbb{P}$ be a forcing poset in $M$. Suppose that $G$ is a $\mbb{P}$-generic filter over $M$, and that $H$ is a $j( \mbb{P} )$-generic filter over $N$ with $j[G] \subseteq H$. It is standard that $j^* : M[G] \to N[H]$ defined by $j^* ( i_G ( \dot{x} ) ) := i_H ( j( \dot{x} ) )$ for each $\dot{x} \in M^\mbb{P}$ is an elementary embedding. (See \cite{C} Proposition 9.1.) We refer to $j^*$ as a natural extension of $j$ by $G$ and $H$. 

Note that if $j$ has width $\leq \kappa$ for an ordinal $\kappa$, then $j^*$ also has width $\leq \kappa$: Take an arbitrary $x \in N[H]$. Let $\dot{x}$ be a $j( \mbb{P} )$-name for $x$. Then there is $f : \kappa \to M$ in $M$ and $\alpha < j( \kappa )$ such that $j(f) ( \alpha ) = \dot{x}$ since $j$ has width $\leq \kappa$. We may assume that $f( \beta )$ is a $\mbb{P}$-name for all $\beta < \kappa$ in $M$. In $M[G]$ define a function $f^*$ on $\kappa$ by $f^* ( \beta ) = i_G ( f( \beta ) )$. Then $j^* ( f^* ) ( \alpha ) = x$.

Now we are ready to prove Proposition \ref{prop:pres_strcmpct}. Throughout the proof, we shall use Fact \ref{fact:Easton_it_basics} repeatedly without explicit reference.

\begin{proof}[Proof of Proposition \ref{prop:pres_strcmpct}]
Suppose that $G_\kappa$ is a $\mbb{P}_\kappa$-generic filter over $V$. Take an arbitrary regular cardinal $\lambda > \kappa$. We prove that $\kappa$ is $\lambda$-strongly compact in $V[ G_\kappa ]$.

We first make some preliminary observations in $V$. In $V$, let $U_0$ be a normal $\kappa$-ultrafilter over $\mcal{P}_\kappa \lambda$, and let $j : V \to M = \mrm{Ult} ( V , U_0 )$ be the ultrapower map. By the assumptions of the proposition and standard calculations, we have the following:
\begin{itemize}
\item[(i)] ${}^\lambda M \subseteq M$ in $V$.
\item[(ii)] $\lambda < ( \lambda^+ )^V = ( \lambda^+ )^M < j( \kappa ) < ( j( \kappa )^+ )^M \leq j( \lambda ) < ( \lambda^{++} )^V$.
\item[(iii)] In $M$ the least measurable cardinal $> \kappa$ is greater than $\lambda$.
\end{itemize}
Note that $\kappa$ is measurable in $M$. In $M$, let $U_1$ be a normal $\kappa$-ultrafilter over $\kappa$ of Mitchell order $0$, and let $k : M \to N = \mrm{Ult} ( M , U_1 )$ be the ultrapower map. Then we have the following:
\begin{itemize}
\item[(iv)] ${}^\kappa N \subseteq N$ in $M$.
\item[(v)] $( \kappa^+ )^M = ( \kappa^+ )^N < k( \kappa ) < ( k( \kappa )^+ )^N < ( \kappa^{++} )^M$.
\item[(vi)] $\kappa$ is not measurable in $N$.
\item[(vii)] $k$ has width $\leq \kappa$.
\end{itemize}
Let $l := k \circ j : V \to N$. Notice that  $\mrm{crit} (l) = \kappa$, and $l( \kappa ) \geq j( \kappa ) > \lambda$. Let $x := k( j[ \lambda ] ) \in N$. Then $l[ \lambda ] \subseteq x \subseteq l( \lambda )$. Moreover, since $| j[ \lambda ] |^M = \lambda$, we have $|x|^N = k( \lambda ) < k( j( \kappa ) ) = l( \kappa )$. Hence, 
\begin{itemize}
\item[(viii)] $l$ is a $( \kappa , \lambda )$-strongly compact embedding.
\end{itemize}

Our goal is to construct, in $V[ G_\kappa ]$, an $l( \mbb{P}_\kappa )$-generic filter $G_{l( \kappa )}$ over $N$,  and to lift $l$ to an elementary embedding from $V[ G_\kappa ]$ to $N[ G_{l( \kappa )} ]$.

Let $( \langle \mbb{P}_\alpha^M : \alpha \leq j( \kappa ) \rangle , \langle \dot{\mbb{Q}}_\alpha^M : \alpha < j( \kappa ) \rangle )$ and $( \langle \mbb{P}_\alpha^N : \alpha \leq l( \kappa ) \rangle , \langle \dot{\mbb{Q}}_\alpha^N : \alpha < l( \kappa ) \rangle )$ be the images of $( \langle \mbb{P}_\alpha : \alpha \leq \kappa \rangle , \langle \dot{\mbb{Q}}_\alpha : \alpha < \kappa \rangle )$ by $j$ and $l$, respectively. Let $\dot{\mbb{P}}_{\alpha , \beta}^M$ and $\dot{\mbb{P}}_{\alpha , \beta}^N$ denote the corresponding tails of $\mbb{P}_\beta^M$ and $\mbb{P}_\beta^N$ after $\alpha$. Then $\mbb{P}_{j( \kappa )}^M$ and $\mbb{P}_{l( \kappa )}^N$ factor in $M$ and $N$, respectively, as follows:
\begin{itemize}
\item[(ix)] $\mbb{P}_{j( \kappa )}^M \sim \mbb{P}_\kappa^M * \dot{\mbb{Q}}_\kappa^M * \dot{\mbb{P}}_{\kappa + 1 , j( \kappa )}^M$.
\item[(x)] $\mbb{P}_{l( \kappa )}^N \sim \mbb{P}_\kappa^N * \dot{\mbb{P}}_{\kappa , k( \kappa )}^N * \dot{\mbb{Q}}_{k( \kappa )}^N * \dot{\mbb{P}}_{k( \kappa ) + 1 , l( \kappa )}^N$.
\end{itemize}

From now on, we work in $V[ G_\kappa ]$ unless otherwise stated. We will construct a $\mbb{P}_{l( \kappa )}^N$-generic filter over $N$ along its factorization (x). First note that $\mbb{P}_\kappa^M = \mbb{P}_\kappa^N = \mbb{P}_\kappa$, and that $G_\kappa$ is $\mbb{P}_\kappa$-generic over both $M$ and $N$. We will use $G_\kappa$ for a generic filter for the first factor $\mbb{P}_\kappa^N$. Also note that GCH holds in $N[ G_\kappa ]$, $M[ G_\kappa ]$ and $V[ G_\kappa ]$.

We now construct generic filters for the second, third and fourth factors of the decomposition in (x). Let $\mbb{P}$ be the interpretation of $\dot{\mbb{P}}_{\kappa , k( \kappa ) +1}^N$ by $G_\kappa$. We claim the following.

\begin{claim}
In $M[ G_\kappa ]$, there is a $\mbb{P}$-generic filter $G_{\kappa , k( \kappa ) + 1}$ over $N[ G_\kappa ]$.
\end{claim}

\noindent
\emph{Proof of Claim 1}.
Let $\mcal{A}_0$ be the set of all maximal antichains of $\mbb{P}$ in $N[ G_\kappa ]$. By (I) and (II) of the proposition, in $N[ G_\kappa ]$, $\mbb{P}$ is $k( \kappa )^+$-c.c.~and we may assume $| \mbb{P} | \leq k( \kappa )^+$. So $| \mcal{A}_0 | \leq k( \kappa )^+$ in $N[ G_\kappa ]$. And by (v), it follows that  $| \mcal{A}_0 | \leq \kappa^+$ in $M[ G_\kappa ]$.

By (vi), in $N$, the least measurable cardinal $\geq \kappa$ is greater than $\kappa^+$. Thus, by (I), (II) and Fact \ref{fact:Easton_it_basics} (2), $\mbb{P}$ is $\kappa^+$-strategically closed in $N[ G_\kappa ]$. Moreover, by (iv) and the $\kappa$-c.c.~of $\mbb{P}_\kappa$, we have  ${}^\kappa N[ G_\kappa ] \subseteq N[ G_\kappa ]$ in $M[ G_\kappa ]$. Hence $\mbb{P}$ is $\kappa^+$-strategically closed also in $M[ G_\kappa ]$. Then in $M[ G_\kappa ]$, since $| \mcal{A}_0 | \leq \kappa^+$, we can take a $\mbb{P}$-generic filter $G_{\kappa , k( \kappa ) + 1}$ over $N[ G_\kappa ]$.
\hfill $\square$(Claim 1)

\medskip
Let $G_{k( \kappa ) + 1} := G_\kappa * G_{\kappa , k( \kappa ) + 1}$. Let $G_{\kappa , k( \kappa )}$ be an $i_{G_\kappa} ( \dot{\mbb{P}}_{\kappa , k( \kappa )}^N )$-generic filter over $N[ G_\kappa ]$ obtained by projecting $G_{k( \kappa ) + 1}$, and let $G_{k( \kappa )} := G_\kappa * G_{\kappa , k( \kappa )}$. Also, let $H_{k( \kappa )}$ be an $i_{G_{k( \kappa )}} ( \dot{\mbb{Q}}_{k( \kappa )}^N )$-generic filter over $N[ G_{k( \kappa )} ]$ naturally obtained from $G_{k( \kappa ) + 1}$.

Let $\mbb{R}$ be the interpretation of $\dot{\mbb{P}}_{k( \kappa ) + 1 , l( \kappa )}^N$ by $G_{k( \kappa ) + 1}$. We claim the following.

\begin{claim}
In $V[ G_\kappa ]$, there is an $\mbb{R}$-generic filter $G_{k( \kappa ) + 1 , l( \kappa )}$ over $N[ G_{k( \kappa ) + 1} ]$.
\end{claim}

\noindent
\emph{Proof of Claim 2}. Let $\mbb{S}^M$ be the interpretation of $\dot{\mbb{P}}_{\kappa , j( \kappa )}^M$ by $G_\kappa$. Thus, $\mbb{S}^M$ corresponds to the last two factors of the factorization (ix). In $M[ G_\kappa ]$, $\mbb{S}^M$ factors as $\mbb{Q}^M * \dot{\mbb{R}}^M$, where $\mbb{Q}^M = i_{G_\kappa} ( \dot{\mbb{Q}}_\kappa^M )$, and $\dot{\mbb{R}}^M$ corresponds to $\dot{\mbb{P}}_{\kappa + 1 , j( \kappa )}^M$. Let $\mbb{A}^M$ be $\mbb{A} ( \mbb{Q}^M , \dot{\mbb{R}}^M )$ in $M[ G_\kappa ]$. The following subclaim is proved by an argument similar to that of Claim 1.

\begin{subclaim*}
In $V[ G_\kappa ]$, there is an $\mbb{A}^M$-generic filter $I^M$ over $M[ G_\kappa ]$.
\end{subclaim*}

\noindent
\emph{Proof of Subclaim}.
In $M[ G_\kappa ]$, there is a dense sub-poset $\mbb{A}'$ of $\mbb{A}^M$ of size $\leq j( \kappa )$, since $| \mbb{Q}^M | < j( \kappa )$ and $\mbb{Q}^M$ forces that $\dot{\mbb{R}}^M$ has a dense subset of size $\leq j( \kappa )$. Let $\mcal{D}_1$ be the set of all dense subsets of $\mbb{A}'$ in $M[ G_\kappa ]$. By GCH, $| \mcal{D}_1 | \leq j( \kappa )^+$ in $M[ G_\kappa ]$. Then, by (ii), it follows that  $| \mcal{D}_1 | \leq \lambda^+$ in $V[ G_\kappa ]$.

By (iii), in $M[ G_\kappa ]$, $\mbb{Q}^M$ forces $\dot{\mbb{R}}^M$ to be $\lambda^+$-strategically closed. Then, by Lemma \ref{lem:term_basic2}, $\mbb{A}^M$ is $\lambda^+$-strategically closed in $M[ G_\kappa ]$. Thus so is $\mbb{A}'$. Moreover,  by (i) and the $\kappa$-c.c.~of $\mbb{P}_\kappa$, we have ${}^\lambda M[ G_\kappa ] \subseteq M[ G_\kappa ]$ in $V[ G_\kappa ]$. Therefore $\mbb{A}'$ is $\lambda^+$-strategically closed in $V[ G_\kappa ]$. Then in $V[ G_\kappa ]$, since $| \mcal{D}_1 | \leq \lambda^+$, there is an $\mbb{A}'$-generic filter over $M[ G_\kappa ]$. This filter generates an $\mbb{A}^M$-generic filter $I^M$ over $M[ G_\kappa ]$.
\hfill $\square$(Subclaim)

\medskip
Next, let $\mbb{S}^N$ be the interpretation of $\dot{\mbb{P}}_{k( \kappa ) , l( \kappa )}^N$ by $G_{k( \kappa )}$, corresponding to the last two factors in (x). In $N[ G_{k( \kappa )} ]$, $\mbb{S}^N$ factors as $\mbb{Q}^N * \dot{\mbb{R}}^N$, where
\begin{itemize}
\item[(xi)] $\mbb{Q}^N = i_{G_{k( \kappa )}} ( \dot{\mbb{Q}}_\kappa^N )$, and $i_{H_{k( \kappa )}} ( \dot{\mbb{R}}^N ) = \mbb{R}$.
\end{itemize}
Here recall that $H_{k( \kappa )}$ is an $i_{G_{k( \kappa )}} ( \dot{\mbb{Q}}_\kappa^N )$-generic filter over $N[ G_{k( \kappa )} ]$. Let $\mbb{A}^N$ be $\mbb{A} ( \mbb{Q}^N , \dot{\mbb{R}}^N )$ in $N[ G_{k( \kappa )} ]$.

Note that $k[ G_\kappa ] \subseteq G_{k( \kappa )}$. Let $k^* : M[ G_\kappa ] \to N[ G_{k( \kappa )} ]$ be the natural extension of $k$. Then $k^* ( \mbb{A}^M ) = \mbb{A}^N$, and $k^*$ has width $\leq \kappa$ by (vii) and the remark immediately preceding the proof of the proposition. Hence, by Lemma \ref{lem:width_generic}, $k^* [ I^M ]$ generates an $\mbb{A}^N$-generic filter $I^N$ over $N[ G_{k( \kappa )} ]$. By (xi) and Lemma \ref{lem:term_basic}, in $V[ G_\kappa ]$, we obtain an $\mbb{R}$-generic filter $G_{k( \kappa ) + 1 , l( \kappa )}$ over $N[ G_{k( \kappa )} ][ H_{k( \kappa )} ] = N[ G_{k( \kappa ) + 1} ]$ from $H_{k( \kappa )}$ and $I^N$.
\hfill $\square$(Claim 2)

\medskip
In $V[ G_\kappa ]$, let $G_{l( \kappa )} := G_{k( \kappa ) + 1} * G_{k( \kappa) + 1 , l( \kappa )}$. Then $G_{l( \kappa )}$ is a $\mbb{P}_{l( \kappa )}^N$-generic filter over $N$, and $l[ G_\kappa ] \subseteq G_{l( \kappa )}$. Let $l^* : V[ G_\kappa ] \to N[ G_{l( \kappa )} ]$ be the natural extension of $l$. By (viii), $l^*$ is a $( \kappa , \lambda )$-strongly compact embedding. Therefore, $\kappa$ is $\lambda$-strongly compact in $V[ G_\kappa ]$.
\end{proof}

%--------------------------------------------------------------------------------------------------------------------------------------------------------------------------------------------------------------------------------------------------------
\subsection{Proof of Main Theorem} \label{sec:main}
%--------------------------------------------------------------------------------------------------------------------------------------------------------------------------------------------------------------------------------------------------------

We now have all the ingredients needed for the proof of the main theorem.

\setcounter{claim}{0}

\begin{proof}[Proof of Main Theorem]
Our forcing construction consists of two major steps. The first step is the forcing of Fact \ref{fact:Apter_Cummings}, due to Apter \& Cummings \cite{AC}. Let $V_1$ denote the corresponding forcing extension of the ground model $V$.

In $V_1$, GCH holds, and $\kappa$ is supercompact. In $V_1$, define
\begin{eqnarray*}
B & := & \{ \alpha < \kappa : \mbox{$\alpha$ is measurable} \} , \\
E & := & \textstyle \ltpt ( \kappa^+ ) \setminus \bigcup_{\alpha \in B} E^{\kappa^+}_{\alpha^+} .
\end{eqnarray*}
Then $\square_\kappa^E$ holds in $V_1$. Recall that there is no measurable cardinal $> \kappa$ in $V$, and that the forcing of Fact \ref{fact:Apter_Cummings} has size $2^\kappa$. Hence,  there is no measurable cardinal $> \kappa$ also in $V_1$ by the well-known fact, due to Levy \& Solovay \cite{LS}, that small forcings create no new measurable cardinals.

The second stage is an Easton support iteration of the forcings $\mbb{S}_\alpha$ from Fact \ref{fact:Shioya}, due to Shioya \cite{Sh}, for $\alpha \in B$. In $V_1$, let $( \langle \mbb{P}_{\alpha} : \alpha \leq \kappa \rangle , \langle \dot{\mbb{Q}}_{\alpha} : \alpha < \kappa \rangle )$ be the Easton support forcing iteration such that
\begin{itemize}
\item if $\alpha \in B$, then $\Vdash_\alpha ``\, \dot{\mbb{Q}}_\alpha = \mbb{S}_\alpha \,"$,
\item if $\alpha \notin B$, then $\Vdash_\alpha ``\, \mbox{$\dot{\mbb{Q}}_\alpha$ is trivial} \,"$.
\end{itemize}
Here note that if $\alpha \in B$, then $\Vdash_\alpha ``\,\mbox{$\alpha$ is inaccessible}\,"$ by Fact \ref{fact:Easton_it_basics}. Hence $\mbb{S}_\alpha$ is defined in $V_1^{\mbb{P}_\alpha}$.

Suppose $G_\kappa$ is a $\mbb{P}_\kappa$-generic filter over $V_1$, and let $V_2 := V_1 [G_{\kappa}]$. We show that $V_2$ is the desired forcing extension of $V$. In what follows, we work in $V_2$.

Note that $\kappa$ is strongly compact by Proposition \ref{prop:pres_strcmpct}. Therefore, by Proposition \ref{prop:downward_Menas_2stat}, it suffices to prove that $\mcal{P}_\kappa \kappa^+$ is not $2$-stationary. For this, it is enough to prove that $S := S( \kappa , \kappa^+ )$ is a non-reflecting $1$-stationary set.

\begin{claim}
$S$ is $1$-stationary.
\end{claim}

\noindent
\emph{Proof of Claim 1}. By Fact \ref{fact:1-stat} (2), it is enough to prove that $S$ is strongly stationary. Take an arbitrary function $f : \mcal{P}_\kappa \kappa^+ \to \mcal{P}_\kappa \kappa^+$. We find $x \in S$ such that $x \cap \kappa$ is inaccessible, and $f[ \mcal{P}_{x \cap \kappa} x] \subseteq \mcal{P}_{x \cap \kappa} x$.

Let $W := ( \mcal{P}_\kappa \kappa^+ )^{V_1}$. Define $f' : W \to \mcal{P}_\kappa \kappa^+$ by $f' (y) := \bigcup_{z \subseteq y} f(z)$. Since $V_2$ is a $\kappa$-c.c.~forcing extension of $V_1$, there is a function $f^* : W \to W$ in $V_1$ such that $f' (y) \subseteq f^* (y)$ for all $y \in W$. Thus, whenever $z \subseteq y \in W$, we have $f(z) \subseteq f^* (y)$.

Since $\kappa$ is supercompact in $V_1$, $T := S( \kappa , \kappa^+ )^{V_1}$ is strongly stationary in $V_1$. Hence we can take $x \in T$ such that $\mu := x \cap \kappa$ is inaccessible in $V_1$, and $f^* [ ( \mcal{P}_\mu x )^{V_1} ] \subseteq ( \mcal{P}_\mu x )^{V_1}$. We show that $x$ is as desired.

Note that $x \in S$ and $x \cap \kappa = \mu$ is inaccessible by Fact \ref{fact:Easton_it_basics}. We must show that $f[ \mcal{P}_\mu x ] \subseteq \mcal{P}_\mu x$. For this, take an arbitrary $z \in \mcal{P}_\mu x$. Let $\alpha := |z| < \mu$ and $G_{\alpha + 1}$ be the restriction of $G_\kappa$ to $\mbb{P}_{\alpha + 1}$. Then $z \in V_1 [ G_{\alpha + 1} ]$ since the tail of $\mbb{P}_\kappa$ after $\alpha + 1$ adds no new sequences of ordinals of length $\alpha$. But $V_1[ G_{\alpha + 1} ]$ is a $\mu$-c.c.~forcing extension of $V_1$. So there is $y \in ( \mcal{P}_\mu x )^{V_1}$ with $z \subseteq y$. Then $f(z) \subseteq f^* (y) \in ( \mcal{P}_\mu x )^{V_1}$. Hence $f(z) \in \mcal{P}_\mu x$.
\hfill $\square$(Claim 1)

\begin{claim}
$d_1 (S)$ is non-stationary in $\mcal{P}_\kappa \kappa^+$.
\end{claim}

\noindent
\emph{Proof of Claim 2}.
First note that $E = \ltpt ( \kappa^+ ) \setminus \bigcup_{\alpha \in B} E^{\kappa^+}_{\alpha^+}$ and that $\square_\kappa^E$ holds (in $V_2$), since $\mbb{P}_\kappa$ preserves all cofinalities by Fact \ref{fact:Easton_it_basics}. Let $S_E$ be the set of all $x \in S$ with $\sup x \in E$. Then, by Lemma \ref{lem:psquare_S}, $S_E$ is non-stationary in $\mcal{P}_\kappa \kappa^+$.

Then $T := S_< ( \kappa , \kappa^+ ) \cup ( S \setminus S_E )$ contains a club subset of $\mcal{P}_\kappa \kappa^+$ by Lemma \ref{lem:S_<_S_>_stat}. So it suffices to prove that $d_1 (S) \cap T = \emptyset$. But $d_1 (S) \cap S_< ( \kappa , \kappa^+ ) = \emptyset$ by Corollary \ref{cor:dS_S<}. Thus, we just have to  show that $d_1 (S) \cap ( S \setminus S_E ) = \emptyset$.

Suppose $x \in S \setminus S_E$, and let $\mu := x \cap \kappa$. By Fact \ref{fact:1-stat} (1), it suffices to prove that $S \cap \mcal{P}_\mu x$ is non-stationary. Since  $x \in S$ , we have  $\cof ( \sup x ) = \mu^+$. Moreover, since $x \notin S_E$, we have  $\sup x \in \bigcup_{\alpha \in B} E^{\kappa^+}_{\alpha^+}$. Thus, $\mu \in B$. Then,  by the construction of $\mbb{P}_\kappa$, we have that $S( \mu , \mu^+ )$ is non-stationary in $\mcal{P}_\mu \mu^+$.

Now, let $\pi : \mu^+ \to x$ be the inverse of the transitive collapse. Then $S \cap \mcal{P}_\mu x = \{ \pi [y] : y \in S( \mu , \mu^+ ) \}$. Therefore, from Remark \ref{rmk:base_set} it follows that $S \cap \mcal{P}_\mu x$ is non-stationary.
\hfill $\square$(Claim 2)

\medskip

By Claim 1, 2 and Corollary \ref{cor:d_1stat}, $\mcal{P}_\kappa \kappa^+$ is not $2$-stationary.
\end{proof}

%--------------------------------------------------------------------------------------------------------------------------------------------------------------------------------------------------------------------------------------------------------
\section*{Acknowledgments}
%--------------------------------------------------------------------------------------------------------------------------------------------------------------------------------------------------------------------------------------------------------

This research was carried out during the second author's long-term visit to the University of Tokyo in 2025. This research visit was supported by the project \emph{Lógica Matemática} PID2023-147428NB-I00, led by Joan Bagaria. Part of the revision and writing of the paper was carried out during the first half of 2026, while the second author was supported by an Ernst Mach grant of OeAD MPC-2025-00829 and  by the Austrian Science Fund (FWF), through Juan P. Aguilera’s project, grants 10.55776/P36837 and 10.55776/STA139. The authors would like to express their sincere gratitude to Joan Bagaria and Juan Aguilera for their kind support to this research. The first author's research was also supported by JSPS KAKENHI Grant Number 24K06828.

%--------------------------------------------------------------------------------------------------------------------------------------------------------------------------------------------------------------------------------------------------------

\bibliographystyle{amsplain}

%--------------------------------------------------------------------------------------------------------------------------------------------------------------------------------------------------------------------------------------------------------

\end{document}